\documentclass[12pt]{amsart}

\usepackage{amsmath, amssymb, mathtools}
\usepackage{hyperref}
 
\newcounter{my_enumerate_counter}
\newcommand{\pushcounter}{\setcounter{my_enumerate_counter}{\value{enumi}}}
\newcommand{\popcounter}{\setcounter{enumi}{\value{my_enumerate_counter}}}

\newcommand{\dminus}{\dot -}
\newcommand{\bbH}{[0,1)}
\newcommand{\bbT}{\mathbb T}
 \DeclareMathOperator{\Cb}{C_b}
 \DeclareMathOperator{\Qf}{Q}
 \DeclareMathOperator{\Aut}{Aut}
 \DeclareMathOperator{\Exh}{Exh}
 \DeclareMathOperator{\dom}{dom}
 
 \DeclareMathOperator{\dist}{dist}
 \DeclareMathOperator{\supp}{supp}
 \DeclareMathOperator{\len}{len}
 \DeclareMathOperator{\Fin}{Fin}

\newcommand{\cY}{\mathcal Y} 
\newcommand{\cS}{\mathcal S} 
\newcommand{\bft}{\mathbf t} 
\newcommand{\bfs}{\mathbf s}

\newcommand{\cP}{\mathcal P}

\newcommand{\cI}{\mathcal I}
\newcommand{\cL}{\mathcal L}

\newtheorem{thm}{Theorem}[section]
\newtheorem{theorem}{Theorem}
\newtheorem{prop}[thm]{Proposition}
\newtheorem{lemma}[thm]{Lemma}

\newtheorem{cor}[thm]{Corollary}
\newtheorem{coro}[theorem]{Corollary}

\newtheorem{conj}[thm]{Conjecture}

\newtheorem{claim}[thm]{Claim}

\theoremstyle{definition}

\newcommand{\rs}{\mathord{\upharpoonright}}

\newcommand{\bb}{{\bf b}}

\newcommand{\bbC}{{\mathbb C}}

\newcommand{\bbN}{{\mathbb N}}

\newcommand{\bbR}{{\mathbb R}}

\newcommand{\bbP}{{\mathbb P}}

\newcommand{\cA}{{\mathcal A}}

\newcommand{\calD}{{\mathcal D}}
\newcommand{\cF}{{\mathcal F}}

\newcommand{\cK}{{\mathcal K}}

\newcommand{\cU}{{\mathcal U}}

\newcommand{\cZ}{{\mathcal Z}}

\newcommand{\range}{{\rm range}}

\newcommand{\e}{\epsilon}

\DeclareMathOperator{\ThE}{Th_\exists}

\title{Rigidity of continuous quotients}

\author{Ilijas Farah}

\address{Department of Mathematics and Statistics, York University, 4700 Keele Street, North York, Ontario, Canada, M3J 1P3, and Matematicki Institut, Kneza Mihaila 35, Belgrade, Serbia}
\email{ifarah@mathstat.yorku.ca}
\urladdr{http://www.math.yorku.ca/~ifarah}

\author{Saharon Shelah}

\address{The Hebrew University of Jerusalem\\
Einstein Institute of Mathematics\\
Edmond J. Safra Campus, Givat Ram\\
Jerusalem 91904, Israel\\ and \\
Department of Mathematics\\
Hill Center-Busch Campus\\
Rutgers, The State University of New Jersey\\
110 Frelinghuysen Road\\
Piscataway, NJ 08854-8019 USA}

\email{shelah@math.huji.ac.il}
\urladdr{http://shelah.logic.at/}

\thanks{The first author was partially supported by NSERC}

\thanks{The second author  would like to thank the Israel Science Foundation, Grant 
no. 710/07, and the National Science Foundation, Grant no. DMS 1101597 for partial support of this research. No. 1042 on Shelah's list of publications.}


\date{\today}

\begin{document}

\begin{abstract}   
We study countable saturation of metric  reduced products and 
introduce continuous fields of metric structures indexed by locally compact,  
separable, completely metrizable spaces. Saturation of the reduced product depends 
both on  the underlying index space and the model.  
%We also provide an (overdue) 
 % proof that the reduced products of metric models corresponding to the Fr\'echet ideal 
 % are countably saturated.  
By using the Gelfand--Naimark duality we conclude
that the assertion ``Stone--\v{C}ech remainder of the half-line has only trivial automorphisms'' 
is independent from ZFC.  
Consistency of this statement follows from the Proper Forcing Axiom
and this is the first known example  of a connected space with this property. 
\end{abstract} 

\maketitle

%\tableofcontents

The present paper has two largely independent parts moving in two opposite directions. 
The first  part (\S\S\ref{S.Countable}--\ref{S.asymptotic})
 uses model theory of metric structures and it is concerned with 
the degree of saturation of various reduced products. The second 
part (\S\ref{S.W})  uses set-theoretic methods and it is 
mostly concerned with rigidity of Stone--\v{C}ech\footnote{Since in Czech alphabet letter `\v{c}' precedes letter `s'
 some authors write \v{C}ech--Stone compactification instead of Stone--\v{C}ech compactification.} 
remainders of locally compact, Polish spaces. 
(A topological space is \emph{Polish} if it is separable and completely metrizable.) 
The two parts are linked by the standard
fact that saturated structures have many automorphisms (the continuous case of this fact is
given in Theorem~\ref{T.Aut}).

By $\beta X$ we denote the Stone--\v{C}ech compactification of $X$ and 
by $X^*$ we denote its remainder (also called corona), $\beta X\setminus X$.  
A continuous map 
$\Phi\colon X^*\to Y^*$ is \emph{trivial} 
if there are a compact subset $K$ of $ X$ and  a continuous map $f\colon X\setminus K 
\to Y$ such that $\Phi=\beta f\rs X^*$, where 
 $\beta f\colon \beta X\to \beta Y$ is the unique continuous extension of $f$. 
Continuum Hypothesis (CH) implies that all Stone--\v{C}ech remainders of
locally compact, zero dimensional, non-compact Polish spaces are homeomorphic. 
This is a consequence of Parovi\v{c}enko's theorem, see e.g., \cite{vM:Introduction}.
By using Stone duality, this follows from the fact that all atomless Boolean algebras are elementarily equivalent and the countable saturation of corresponding reduced products.    
The latter also hinges on the fortuitous fact that the theory of atomless Boolean algebras admits  elimination of quantifiers.   
See also 
  \cite{DoHa:Universal} where  similar 
  model-theoretic methods were applied to the lattice of closed subsets of a Stone--\v{C}ech remainder.  

  Gelfand--Naimark duality (see e.g., \cite{Black:Operator}) 
associates autohomeomorphisms of a compact Hausdorff space $X$ to automorphisms
of the C*-algebra $C(X)$ of continuous complex-valued functions
on $X$. Logic of metric structures (\cite{BYBHU}), or rather its version adapted to 
C*-algebras (\cite{FaHaSh:Model2}) is applied to analyze these algebras. 
The idea of defining autohomeomorphism of a compact Hausdorff space $X$
indirectly via  an automorphism of $C(X)$ dates back at least to the  discussion 
in the introduction of~\cite{PhWe:Calkin}. 

The study of saturation properties of corona algebras was initiated in 
 \cite{FaHa:Countable} where it was shown that 
all coronas of separable C*-algebras satisfy a restricted form of saturation, the so-called countable degree-1 saturation. 
Although it is not clear whether degree-1 saturation suffices to construct many automorphisms, 
 even this restricted notion has 
interesting consequences (\cite[Theorem~1]{FaHa:Countable}, 
\cite[Theorem~8]{choi2013nonseparable}, \cite[\S 2]{Voi:Countable}). 
We also note that most coronas are not $\aleph_2$-saturated provably in ZFC. 
For $C( \omega^*)$\footnote{We use  the notation commonly accepted in set theory and 
  denote the least infinite ordinal (identified with the set of natural numbers including zero) by $\omega$.} this is a consequence of 
Hausdorff's construction of a gap in $\cP(\omega)/\Fin$, and for the Calkin 
algebra this was proved in \cite{Za-Av:Gaps}.

 It was previously known that 
CH  implies the existence of nontrivial autohomeomorphisms  
of the Stone--\v{C}ech remainder  of $\bbH$ (this is a result of Yu, 
see \cite[\S 9]{Ha:Cech-Stone}).

 Definitions of types and countable saturation are reviewed in \S\ref{S.Countable}.

\begin{theorem} \label{T.I.1} C*-algebra $C(\bbH^*)$ is countably saturated. 
\end{theorem}

This is a consequence of Theorem~\ref{T2}, where sufficient conditions for a  quotient 
continuous field of models indexed by a Stone--\v{C}ech remainder of a locally compact Polish space
to be countably saturated are provided. See also Proposition~\ref{P.NS} where the necessity of 
some conditions of Theorem~\ref{T2} was shown. 
We also prove countable saturation of 
 reduced products of metric structures 
corresponding to the Fr\'echet ideal (Theorem~\ref{T.MT.1}) and 
so-called layered ideals (Theorem~\ref{T3}). More general metric reduced products are considered in 
\S\ref{S.Generalizations} where a model-theoretic interpretation of a result of \cite{JustKr} is given.

We note that Theorem~\ref{T.I.1} implies a strengthening of a result of Yu  (see \cite[\S 9]{Ha:Cech-Stone}).

\begin{coro} 
Continuum Hypothesis implies that $C(\bbH^*)$  has $2^{\aleph_1}$ automorphisms and  
that $\bbH^*$ has $2^{\aleph_1}$ autohomeomorphisms.  
In particular it  implies that $[0,1)^*$  has nontrivial automorphisms. 
\end{coro} 

\begin{proof} 
By Theorem~\ref{T.I.1} and CH $C(\bbH^*)$ is saturated. 
By Theorem~\ref{T.Aut} it has $2^{\aleph_1}$ automorphisms. 
Gelfand--Naimark duality implies that $\bbH^*$ has $2^{\aleph_1}$ automorphisms. 
Finally, CH  implies that $2^{\aleph_0}<2^{\aleph_1}$ and 
there are only $2^{\aleph_0}$ continuous functions from $[0,1)$ to itself. 
\end{proof}

Let us now consider the situation in which quotient structures are maximally rigid. 
The first result in this direction was the second author's result that consistently with ZFC 
all autohomeomorphisms of $\omega^*$ are trivial (\cite[\S IV]{Sh:Proper}). 
PFA implies that all homeomorphisms between Stone--\v{C}ech remainders of locally compact  Polish spaces that are in addition 
countable or zero-dimensional are trivial  
(\cite[\S 4.1]{Fa:AQ}, 
\cite[Theorem 4.10.1]{Fa:AQ}, and \cite{FaMcK:Homeomorphisms}). 
The effect of the Proper Forcing Axiom (PFA) to quotient structures extends to the non-com\-mu\-ta\-tive context; 
see the discussion at the beginning of \S\ref{S.Continuous.Field}
as well as  
 \cite{Fa:All},  \cite{McK:Reduced}  and \cite{Gha:FDD}. 
 All of these results, as well as the following theorem, appear to be instances of a hitherto unknown general result (see  
 \cite{Fa:Rigidity}). 
 
\begin{theorem}[PFA] \label{T0} Every autohomeomorphism  of $\bbH^*$ is trivial. 
\end{theorem} 

We  prove a more general result, Theorem~\ref{T1}, as   a step towards proving that 
all Stone--\v{C}ech remainders of locally compact Polish spaces have only `trivial' automorphisms
assuming PFA.  
An inspection of its proof shows that it uses only consequences of PFA whose consistency does not require large cardinal axioms. 

The proof of Theorem~\ref{T0} introduces a novel technique. 
In all previously known cases  rigidity of the remainder $X^*$ was proved 
 by  representing  it as an  inverse limit of spaces homeomorphic to~$\omega^*$ (see e.g., 
 \cite[\S 4]{Fa:AQ}).  This essentially applies even to the non-commutative case, where
 the algebras were always presented as direct limits of algebras with an abundance of projections. 
 This approach clearly  
 works only in the case when $X$ is  zero-dimensional (or, in the noncommutative case, 
 when the C*-algebra has real rank zero) and our proof of Theorem~~\ref{T0} necessarily takes a different route.

\subsection*{Organization of the paper} 
In \S\ref{S.Countable} we review conditions, types, and 
 saturation of metric structures. In \S\ref{S.Reduced}, ,
it is proved that reduced product with respect to the Fr\'echet ideal is always countably saturated. 
Proof of Theorem~\ref{T.MT.1} proceeds via discretization of ranges of metric formulas given in \S\ref{S.Deconstructing}
and is completed in \S\ref{S.T.MT.1}.
Models $C_b(X,A)$ and $C_0(X,A)$ are introduced in 
  \S\ref{S.Continuous.Field} and countable saturation of the corresponding quotients  under 
  additional assumptions is 
  proved in Theorem~\ref{T2}, whose 
  immediate consequence is Theorem~\ref{T.I.1}. Proposition~\ref{P.NS} 
  provides some limiting examples. 
 \S\ref{S.Rep} is independent from the rest of the paper. In it  we prove
 Theorem~\ref{T0} by using set-theoretic methods. 
 We conclude with some brief remarks in \S\ref{S.CR}.  
 
\subsection*{Notation} 
If $a\subseteq \dom(h)$ we  write $h[a]$ for the pointwise image of~$a$. 
An element $a$ of a product $\prod_n A_n$ is always identified with the sequence 
$(a_n: n\in \omega)$; in particular indices in subscripts are usually used for this purpose. 
For a set $A$ we denote its cardinality by $|A|$. 
In some of the literature (e.g., \cite{DoHa:Universal} or \cite{Ha:Cech-Stone})
the half-line is denoted by ${\mathbb H}$. Since the same symbol is elsewhere used to denote the 
 half-plane, we avoid using it. In our  results about  Stone--\v{C}ech remainders $[0,1)$ can 
 be everywhere replaced with ${\mathbb H}$.  
 A subset $D$ of a  metric space is  \emph{$\e$-discrete} if $d(a,b)\geq \e$ for all distinct $a$ 
 and $b$ in $D$. 
We also follow \cite{BYBHU}  and write $x\dminus y$ for $\max(x-y,0)$. 

\subsection*{Acknowledgments} 
The proof of Theorem~\ref{T.MT.1} was inspired by   conversations 
of the first author with Bradd Hart over the past several years. 
Also, the included proof of Theorem~\ref{T.Aut} was communicated 
to the first author and David Sherman in an email from Bradd Hart in June 2010. 
We would like to thank Bradd for his kind permission to include this proof. 
The first author would also like to thank Bruce Blackadar for a useful remark 
on coronas and  N. Christopher Phillips for several remarks on an early draft of this paper. 
We would also like to thank  Alessandro Vignati
and the anonymous referee for making a number of very helpful suggestions. 
After this paper was completed Isaac Goldbring pointed out that 
reduced products of metric structures were also  studied in \cite{lopes2013reduced}.

%%%%%%%%%%%%%%%%%%%%%%

\section{Countable saturation} 
\label{S.Countable}
\subsection{Conditions, types and saturation} 
A quick review of the necessary model-theoretic background is in order; 
see \cite{BYBHU} and  \cite{Ha:Continuous}
%\cite[\S\S 4.3--4.4]{FaHaSh:Model2} 
for more details.  Our motivation comes from  study of saturation properties
of C*-algebras (\cite{FaHaSh:Model2}, \cite{FaHa:Countable}), 
but we prove novel results for  general metric structures. 
 Fix  language $\cL$ in the logic of metric structures  whose variables are 
 listed as  $\{x_n: n\in\omega\}$. 
In the ensuing discussion we shall write $\bar x, \bar y, \bar a, \dots$ to denote  tuples
of unspecified length and sort. In most interesting cases all entries of the tuple will belong to a single sort, such as the unit ball of the C*-algebra under the consideration, and we shall suppress discussion 
of sorts by assuming all variables are of the same sort.

For a metric formula $\phi(\bar x)$,  metric structure $A$ of the same signature  and  tuple $\bar a$
in $A$ of the appropriate sort, by $\phi(\bar a)^A$ we denote the interpretation (i.e., evaluation) of $\phi(\bar x)$ at $\bar a$ in structure $A$. 
A (closed) 
\emph{condition} is an expression of the form $\phi(\bar x)=r$ for a formula $\phi(\bar x)$ and a real number $r$. We consider conditions over a model $A$, in which case $\phi$ is allowed to have 
elements from $A$ as parameters. Formally, we expand  the language by adding constants  for 
these elements; for details see \cite{BYBHU} or \cite[\S 2.4.1]{FaHaSh:Model2}. 
An \emph{$n$-type}  is a set of conditions all of whose  
free variables are included in the set $\{x_0,\dots, x_{n-1}\}$.
 We shall suppress $n$ throughout and write $x$ instead of $x_0$ if $n=1$. 
In general, a \emph{type} over a model $A$ is a set of conditions with parameters from $A$. 
 An $n$-type $\bft(\bar x)$ 
is \emph{realized} in $A$ if some $n$-tuple $\bar a$ in $A$ we have that $\phi(\bar a)^A=r$ for all conditions 
$\phi(\bar x)=r$ in $\bft(\bar x)$. 
A type is  \emph{consistent} (or \emph{finitely approximately realizable} in the terminology of \cite{FaHa:Countable})
 if every one of its finite subsets can be realized up to an arbitrarily small $\e>0$. 

If $\kappa$ is an infinite cardinal, we say that 
a model $A$ is \emph{$\kappa$-saturated} if every consistent type $\bft$ over $A$  with fewer than $\kappa$ conditions 
is realized in $A$. If the density character of $A$ is $\kappa$ then $A$ 
is \emph{saturated}.  
Instead of \emph{$\aleph_1$-saturated} we shall usually say  \emph{countably saturated}.
Saturated models have remarkable properties. Every saturated model of density character 
$\kappa$  has $2^\kappa$ 
automorphisms, and two saturated models of the same language and same 
character density are isomorphic if and only if they have the same theory
(see any standard text on model theory, e.g., 
\cite{ChaKe},  \cite{Hodg:Model} or \cite{Mark:Model}). 

Following \cite{FaHa:Countable} one may consider restricted versions of saturation. 
If all  consistent quantifier-free types of cardinality $<\kappa$ over a model are realized in it, 
 the model is said to be \emph{quantifier-free $\kappa$-saturated}. 
In case of C*-algebras a weaker notion of degree-1 countable saturation was considered in 
\cite{FaHa:Countable}. C*-algebra $C$ is \emph{countably degree-1 saturated} 
if every type consisting of conditions of the form $\|p\|=r$, where $p$ is a sum of monomials of the form  
 $a$, $axb$ or $ax^*b$ for $a,b$ in $C$ and variable $x$
if and only if it is realizable in $C$. 
All coronas of separable C*-algebras  have this property, 
and it is strong enough to imply many of the known properties of such coronas
(\cite{FaHa:Countable}; see also \cite{EaVi:Degrees}).  

However,  the existence  of saturated models of unstable theories requires nontrivial 
assumptions on cardinal arithmetic, such as the Continuum Hypothesis (see \cite[\S 6]{ChaKe}). 
Nevertheless, in a situation where  focus is on  separable objects, 
countable saturation is sufficient. The fact that the ultrapowers as well 
as the relative commutants of separable subalgebras in ultrapowers are countably saturated 
(see e.g., \cite{FaHaSh:Model2}) is 
largely responsible for their usefulness in the study of separable C*-algebras. 
Saturation of all ultrapowers  of a separable model of an unstable theory 
associated to a nonprincipal ultrafilter on $\omega$  
is equivalent to the Continuum Hypothesis (\cite{FaHaSh:Model1};  see
also  \cite{FaSh:Dichotomy} for a quantitative strenghtening).

Metric formula $\phi(\bar y)$ is in  \emph{prenex normal form} 
if it is of the following form 
for some  $n$ and $k$  
($\bar x$ stands for  $(x(0), \dots, x(2n-1))$)
\[
\Qf_{x(0)} \Qf_{x(1)} \dots \Qf_{x(2n-2)} \Qf_{x(2n-1)} 
f(\alpha_0(\bar x, \bar y), \dots \alpha_{k-1}(\bar x, \bar y))
\]
where each $\Qf$ stands for $\sup$ or $\inf$,  
$f$ is a continuous function and $\alpha_i$ for $i<k$ are atomic formulas.

\begin{lemma}\label{L.PNF} Every type  is equivalent to a type  
such that 
all  formulas occurring in its conditions  are in prenex normal form. 
\end{lemma} 

\begin{proof} This is an easy consequence of \cite[Proposition~6.9]{BYBHU}, 
which states that every formula can be uniformly approximated by formulas in the 
prenex normal form and its proof. As pointed out in this proof, 
 connectives $\dminus$ and $|\cdot |$ 
are monotonic in all of their arguments and therefore standard proof that a 
(discrete) formula is equivalent to one in prenex normal form 
applies to show that if $\phi$ and $\psi$ are in prenex normal 
form then $\phi\dminus \psi$ and $|\psi|$ are equivalent to formulas in 
prenex normal form. 

Fix a 
 condition $\phi(\bar x)=r$.  
For every $n$ fix a formula $\phi_n$ in  prenex normal form such that 
\[
\sup_{\bar x}|\phi_n(\bar x)-\phi(\bar x)|\leq 1/n. 
\]
The type consisting of conditions 
\[
|\psi(\bar x)\dminus r|\dminus 1/n=0
\]
for 
$n\geq 1 $
is equivalent to condition $\phi(\bar x)=r$ and by the above formula $|\psi(\bar x)\dminus r|\dminus 1/n$ is equivalent to a formula in prenex normal form. 
\end{proof}

\subsection{Reduced products over the Fr\'echet ideal}  \label{S.Reduced} 
Fix a language $\cL$ in the logic of metric structures 
  with a distinguished constant symbol~$0$. 
 Assume $A_n$, for $n\in \omega$ are $\cL$-structures. 
 We form two $\cL$-structures as follows (see also \cite[\S 2]{lopes2013reduced} for more details) 
 \begin{align*} 
\textstyle \prod_n A_n=&\{\bar x: (\forall n) x_n\in A_n
\text{ and all $x_n$ belong to the same domain}\},\\
 \textstyle\bigoplus_n A_n=&\textstyle\{\bar x\in \prod_n A_n: (\forall n) x_n\in A_n\text{ and } 
  \limsup_n d_n(x_n,0^{A_n})=0\}. 
 \end{align*} 
  We shall write 
  \[
\textstyle  A_\infty:=\prod_i A_i. 
\]
   Both structures  are considered with respect to the sup metric, in which they are both complete. 
Interpretations of function symbols are defined in the natural way, pointwise. For a tuple 
$\bar a=(\bar a_i: i\in \omega)$ in one of these
models and a function symbol $f$ of the appropriate sort we have
\[
f(\bar a):=(f(\bar a_n):n\in \omega).  
\]
The interpretation of relational symbols is admittedly less canonical. 
If $\bar a$ is a tuple in one of the above models and $R$ is a predicate symbol of the appropriate sort, 
then we define
\[
R(\bar a):=\sup_n  R(\bar a_n). 
\]
This  agrees with our choice of the sup metric on the product space. 
It is also compatible with the convention adopted in the discrete logic, where $R(\bar a)$ is true if
and only if $R(\bar a_n)$ is true for all $n$.

Consider the quotient structure
\[
\cA=\textstyle\prod_n A_n/\bigoplus_n A_n
\]
(also denoted by $\prod_{\Fin} A_n$)
 with the  quotient map 
\[
\textstyle q\colon \prod_n A_n\to \prod_n A_n/\bigoplus_n A_n.
\]  
We shall use same notation for the natural extension of $q$ to  $k$-tuples
 for $k\geq 1$.
The interpretation of an  $\cL$-sort in $\cA$ is the $q$-image of the interpretation of this sort in $\prod_n A_n$.  
On this quotient structure define metric $d$ via
\[
d(q(\bar a), q(\bar b))=\limsup_n d(a_n,b_n). 
\]
For a predicate symbol $R$ and $\bar a$ such that $q(\bar a)$ is of the appropriate sort define
\[
R(q(\bar a))=\limsup_n R(a_n), 
\]
and for a function symbol $f$ and $\bar a$ such that $q(\bar a)$ is of the appropriate sort define
\[
f(q(\bar a))=q(f(\bar a)). 
\]

\begin{claim} Interpretations of function and predicate symbols are well-defined and 
have the correct moduli of uniform continuity. 
\end{claim} 

\begin{proof} Since all proofs are similar, we shall prove the claim only for function symbol $f$. 
Fix $\e>0$ and $\Delta(\e)>0$ such that each $A_n$ satisfies 
\[
d(x,y)<\Delta(\e) \Rightarrow 
d(f(x), f(y))\leq \e. 
\]
Fix $\bar a$ and $\bar b$ in $\prod_n A_n$ such that $d(q(\bar a), q(\bar b))<\Delta (\e)$. 
Then the set  
\[
Z=\{n : d(a_n,  b_n)\geq \Delta(\e)\}
\]
is finite and for all $n\in \omega\setminus Z$ we have 
$d(f(a_n), f(b_n))\leq \e$. 
Therefore $d(f(q(\bar a), q(\bar b))\leq \e$. 
\end{proof}

We record two immediate consequences of definitions in order
to furnish the intuition (see also \cite[\S 2]{lopes2013reduced}). 

\begin{lemma} \label{L.Atomic} 
If $\alpha(\bar x)$ is an atomic formula and $\bar a\in \prod_n A_n$ is of the appropriate sort, then 
$\alpha(\bar a)^{\prod_n A_n}=\sup_n \alpha(\bar a_n)^{A_n}$. \qed
\end{lemma}

\begin{lemma}\label{L.Product} 
For a tuple $\bar a$ in $A_\infty$ and a formula $\phi(\bar x)$ of the appropriate sort 
we have 
\[
\phi(q(\bar a))=\lim_m \lim_n \tilde\phi(\bar a)^{A_{[m,n)}}. 
\]
 (with 
$A_{[m,n)}$ denoting $ \prod_{i=m}^{n-1} A_i$). \qed
\end{lemma}

 In \cite{CoFa:Automorphisms} the following result was alluded to 
 as `obviously true.' It may not be obvious, but at least it is true. 
  
 \begin{thm} \label{T.MT.1} 
 Every reduced product corresponding to the Fr\'echet ideal, 
 $\prod_n A_n/\bigoplus_n A_n$, is countably 
 saturated. 
 \end{thm} 

The proof of Theorem~\ref{T.MT.1} given in  \S\ref{S.T.MT.1}
is  subtler than the proof of the corresponding result for the classical 
`discrete' logic (see e.g., \cite{Hodg:Model}), although  it proceeds by `discretizing' ranges of  
formulas (see~\S\ref{S.Deconstructing}). 

%The counterpart of Theorem~\ref{T0}, Theorem~\ref{T.I.1} (countable saturation of 
%the corona of $C(\bbH)$) is a consequence of a more general result involving saturation of continuous fields of models; see \S\ref{S.Continuous.Field} for definitions and general result, Theorem~\ref{T2}.  
Saturation  of reduced products over ideals other than the Fr\'echet ideal
 will be considered in Theorem~\ref{T3} below.

\subsection{Deconstructing formulas} \label{S.Deconstructing} 
In this subsection we prepare 
the grounds for the proof of Theorem~\ref{T.MT.1}. 
For a bounded $D\subseteq \bbR$ and $n\geq 1$ let 
\[
F(D,n)=\{k 2^{-n-1}: D\cap ((k-1)2^{-n-1},(k+1)2^{-n-1})\neq \emptyset\}. 
\]
Proof of the following lemma is straightforward. 
\begin{lemma}\label{L.pattern} 
Assume $D\subseteq \bbR$ is bounded and $n\in \omega$. 
Then 
\begin{enumerate}
\item  $F(D,n)$ is finite 
\item every element of $D$ is within $2^{-n}$ of an element of $F(D,n)$. 
\item $F(D,n+1)$ uniquely determines $F(D,n)$. 
\item If $D$ is compact then 
\[
\{E\subseteq \bbR: E \text{ is compact and }F(D,n)=F(E,n)\}
\]
is an open neighbourhood of $D$ in the Hausdorff metric. \qed
\end{enumerate}
\end{lemma} 

Let $\phi(\bar y)$ be  a formula in prenex normal form. By adding dummy variables if necessary,  for some $n$ and $k$ 
we can represent $\phi(\bar y)$ as 
\begin{equation}
\label{Eq.PNF}
\adjustlimits \sup_{x(0)} \inf_{x(1)} \dots \adjustlimits \sup_{x(2n-2)} \inf_{x(2n-1)} 
f(\alpha_0(\bar x, \bar y), \dots \alpha_{k-1}(\bar x, \bar y))
\tag{PNF}
\end{equation}
where $f$ is a continuous function and $\alpha_i$ for $i<k$ are atomic formulas, 
possibly with parameters in a fixed model $A$.  

If $\alpha$ is an atomic formula then we let $F(\alpha, n):=F(D,n)$ where $D$ is the set of all possible values of $\alpha$.  Recall that $D$ is always compact. This is true even if $A$ is  a C*-algebra and $\alpha(x)$ is a *-polynomial 
 because the syntax requires variable $x$ to range over a fixed bounded ball (see~\cite{FaHaSh:Model2}).

If $\phi$, $n$, $k$  and $\alpha_i$ for $i<k$, are  as in (PNF)   and $\bar a$ is a tuple in $A$ of the appropriate sort then we define 
 \emph{$n$-pattern of $\phi(\bar a)$} in $A$ to be 
\begin{align*} 
\textstyle P(\phi, \bar a,n)^A:=&\{\bar r\in \prod_{i=0}^{k-1} F(\alpha_i, n): \\ 
&(\forall x(0))(\exists x(1)) \dots (\forall x(2n-2)) (\exists x(2n-1)) \\
&\qquad\max_i|\alpha_i(\bar x, \bar a)^A-r_i|< 2^{-n}\}. 
\end{align*} 
All $x(j)$ range over the relevant domain in $A$.

Assume $\phi_0(\bar y), \dots, \phi_{m-1}(\bar y)$ are as in (PNF)  
with  free variables $\bar y$ of the 
same sort and $\bar a$ is a tuple in $A$ of the appropriate sort. 
 Then we define the \emph{$n$-pattern
 of $\phi_0(\bar a), \dots, \phi_{m-1}(\bar a)$ in $A$} (or \emph{the $n$-pattern of $\phi_0, \dots, \phi_{m-1}$ and $\bar a$ in $A$})
 to be the set
 \[
\textstyle \prod_{i=0}^{m-1} P(\phi_i, \bar a,n)^A. 
 \]

 \begin{lemma} \label{L.types} For all $n\geq 1$ and every language $\cL$, 
 a tuple of $\cL$-formulas in prenex normal form  
 has at most finitely many possible  distinct $n$-patterns in all $\cL$-structures. 
 \end{lemma} 
 
 \begin{proof} The range of every atomic formula is a compact subset of $\bbR$, 
 and therefore every $n$-pattern of an $m$-tuple of formulas
 is a finite subset of 
 \[
 \prod_{i=0}^{m-1} \{k/n: |k|\leq K\}
 \]
 for some fixed $K<\infty$. 
 \end{proof}

 Given a formula $\phi(\bar x)$ as in (PNF),   a tuple $\bar a$ of elements in $\prod_n A_n$
 of the appropriate sort and $n\geq 1$, consider the set of 
 all $n$-patterns of the form $P_i=P(\phi,\bar a_i,n)^{A_i}$. 
An $n$-pattern $P$ in this set is \emph{relevant} (for $\phi$ and $\bar a$ in $\prod_n A_n$) 
if  it is equal to $P_i$ for infinitely many $i$.

 \begin{lemma} \label{L2} 
 For every formula $\phi(\bar x)$ in prenex normal form  
 and every  $\e>0$ there is $m$ such that 
 for every quotient structure of the form $\prod_n A_n/\bigoplus_n A_n$
 and tuple $\bar a\in \prod_n A_n$ of the appropriate sort 
 the value $\phi(q(\bar a))^{\cA}$ is determined up to $\e$ 
 by the set of all relevant $m$-patterns for $\phi$ and $\bar a$ in $\prod_n A_n$. 
\end{lemma} 
 
 \begin{proof} By adding dummy variables if necessary, we may assume that there exist  natural numbers $n$ and $k$  
and atomic formulas $\alpha_i$, $\alpha_{k-1}$ such that 
 $\phi'(\bar x)$ is of the form 
\[
\adjustlimits \sup_{x(0)} \inf_{x(1)} \dots \sup_{x(2n-2)} \inf_{x(2n-1)} 
f(\alpha_0(\bar x, \bar y), \dots \alpha_{k-1}(\bar x, \bar y))
\]
 Lemma~\ref{L.Atomic} implies that 
 \[
 \alpha_n(\bar x, \bar a)^{A_\infty}=\sup_i \alpha_n(\bar x_i, \bar a_i)^{A_i}
 \]
and therefore     
$\phi'(\bar a)^{A_\infty}\geq r$ if and only if for every $\delta>0$  we have 
(writing $\bar x=(x(0), \dots,  x(2n-1))$ and $\bar x_i=(x(0)_i, \dots , x(2n-1)_i)$ for $i\in \omega$) 
\begin{multline*} 
\exists {x(0)} \forall {x(1)} \dots \exists {x(2n-2)} \forall {x(2n-1)} \\
f(\sup_i \alpha_0(\bar x_i, \bar a_i)^{A_i}, \dots \sup_i \alpha_{k-1}(\bar x_i, \bar a_i)^{A_i})\geq r-\delta.  
\end{multline*}
Hence  
$\phi'(q*\bar a))^{\cA}\geq r$ if and only if for every $\delta>0$ and every $l\in \omega$  we have 
\begin{multline*} 
\exists {x(0)} \forall {x(1)} \dots \exists {x(2n-2)} \forall {x(2n-1)} \\
 f(\sup_{i\geq l} \alpha_0(\bar x_i, \bar a_i)^{A_i}, \dots 
\sup_{i\geq l} \alpha_{k-1}(\bar x_i, \bar a_i)^{A_i})\geq r-\delta.   
\end{multline*}
 Since $\phi$ is equipped with a fixed modulus of uniform continuity
 we can choose $m$ such that changing the value of each  $\alpha_i$
 by no more than $2^{-m}$ does not affect the change of  value of $\phi$ by  more than $\e/2$. 

Therefore the value of $\phi(q(\bar a))^{\cA}$ up to $\e$ depends only 
on the set of relevant $m$-patterns for $\phi$ and  $\bar a$  in $\prod_n A_n$. 
\end{proof}

\subsection{Proof of Theorem~\ref{T.MT.1}}
\label{S.T.MT.1}
 Let $\bft=\{\phi_i(\bar y)=r_i: i\in \omega\}$ be a type with parameters in $\cA$ 
consistent with its theory. By Lemma~\ref{L.PNF} we may assume that all $\phi_i$ are given in prenex normal form. 

By lifting parameters of $\phi_i(\bar y)$ from the quotient 
 $\cA$ to $A_\infty$  we obtain formulas
 $\tilde\phi_i(\bar y)$, 
for $i\in \omega$. Since $\bft$ is consistent, for every $m$ we can choose $\bar a(m)\in A_\infty$ 
so that 
\[
|\phi_i(q(\bar a(m)))^{\cA}-r_i|\leq 2^{-m}
\]
 for all $i\leq m$. 
Fix $m$ for a moment and let 
\begin{align*}
R(m,n)=\{ P: &P\text{ is a relevant $n$-pattern}\\
&\textstyle\text{of  $\tilde\phi_0, \dots, \tilde\phi_{m-1}$ and $\bar a(m)$  in $\prod_i A_i$}\}. 
\end{align*}
Recursively choose increasing sequences  $l(i)$ and $m(i)$  for $i\in \omega$  and infinite sets
$\omega=X_0\supseteq X_1\supseteq\dots$  so that for every $k$ and all $m$  in $X_k$ 
we have the following.  
\begin{enumerate}
\item\label{I.MT.1}  $R(m,k)=R(m(k),k)$. 
\item The set of all patterns of 
  $\tilde\phi_0(\bar a(m(k))_j), \dots, \tilde\phi_{n-1}(\bar a(m(k))_j)$   occurring in $A_j$
  for some $l(m)\leq j <l(m+1)$ is equal to $R(m,k)$. 
  \end{enumerate}
  Let us describe the  construction of these objects. 
  Assume that $k\in \omega$ is such that $l(i)$, $m(i)$  and $X_i$ for $i\leq k$ are as required.  
  Since  there are only finitely many relevant patterns
   we can find an infinite 
  $X_{k+1}\subseteq X_k$ and  $R$ such that $R(m,k+1)=R$ for all $m\in X_{k+1}$. 
  Let $m(k+1)=\min (X_{k+1}\setminus k)$. 
  Now increase $l(k)$ if necessary to assure that 
  all patterns $P_j(m(k+1),k)$ for $j\geq l(k)$ are relevant. 
  Finally, choose $l(k+1)$ large enough so that 
  \[
  \{P_j(m(k),k): l(k)\leq j< l(k+1)\}=R(m(k), k). 
  \]
(This is done with the understanding that $l(k+1)$ may have to be suitably 
increased once $m(k+2)$  is chosen and that this change is innocuous.)

Once all of these objects are chosen define $\bar a\in A_\infty$ via 
\[
\bar a_i=\bar a(m(k))_i, \text{ if } l(k)\leq i <l(k+1). 
\]
The salient property of $\bar a$ is that for every $n$ the set of  relevant $n$-patterns
for $\tilde\phi_0(\bar a), \dots, \tilde\phi_{n-1}(\bar a)$  
is equal to $R(m(k), n)$ for all but  finitely many $k$. 
We claim that $q(\bar a)$ realizes  type $\bft$ in $\cA$. 

Fix $i$ and $\e>0$.  By Lemma~\ref{L2} and the construction of $\bar a$, 
for all large enough $m$ we have that $|\phi_i(q(\bar a))^{\cA}-\phi_i(q(\bar a(m))^{\cA}|<\e$. 
Therefore for every $i$ we have 
$\phi_i(q(\bar a))^{\cA}=r_i$, and $q(\bar a)$ indeed realizes $\bft$.

\section{Countable saturation of other reduced products} 
\label{S.Countable-other} 
We extend Theorem~\ref{T.MT.1} in two  different directions. 
In \S\ref{S.Continuous.Field} continuous fields of models are introduced
with an eye on Theorem~\ref{T.I.1}. In \S\ref{S.Generalizations} we consider 
more traditional reduced products over arbitrary ideals on~$\omega$.

\subsection{Continuous reduced products} 
\label{S.Continuous.Field}

We consider  saturation of continuous reduced products of the form 
$\Cb(X,A)/C_0(X,A)$ (see below for the definitions), and 
the main result section is Theorem~\ref{T2} that has  Theorem~\ref{T.I.1} as a consequence. 

By the Gelfand--Naimark duality (see e.g., \cite{Black:Operator}), 
the categories of compact Hausdorff spaces and unital abelian C*-algebras are equivalent.  
In particular, for a locally compact Polish space  $X$ 
there is a bijective correspondence between autohomeomorphisms
of Stone--\v{C}ech remainder $X^*$ and automorphisms of the 
quotient C*-algebra $\Cb(X)/C_0(X)\cong C(X^*)$. 
Here $\Cb(X)$ is the algebra of all bounded functions $f\colon X\to \bbC$
and $C_0(X)$ is its subalgebra of functions vanishing at  infinity. Since $C_b(X)$ is naturally isomorphic to $C(\beta X)$, 
this quotient algebra is isomorphic to the \emph{corona} (also called the 
\emph{outer multiplier algebra}) of the algebra $C_0(X)$.

Hence results  about  autohomeomorphisms of Stone--\v{C}ech remainders 
$X^*$ are special cases of  results about  automorphisms
of coronas of separable, nonunital C*-algebras. A tentative and very inclusive definition of 
\emph{trivial automorphisms} of the latter was  
given in \cite[Definition~1.1]{CoFa:Automorphisms}. 
 In the present paper we shall be concerned with the 
abelian case only.

If $A$ is a  metric structure and $X$ is a Hausdorff topological space
consider the space
\begin{multline*}
\Cb(X,A)=\{f\colon X\to A: f\text{ is continuous}\\
\text{ and its range is included in a single domain of $A$}\}. 
\end{multline*}
Hence every $f\in \Cb(X,A)$ has a well-defined sort and domain (see \cite[\S 2]{FaHaSh:Model2}) and we write 
$\cS^{\Cb(X,A)}=\{f\in \Cb(X,A): f$ is of sort $\cS\}$. 
Equip $\Cb(X,A)$ with  the metric
\[
d(f,g)=\sup_{x\in X} d(f(x),g(x)),
\]
and interpret predicate and function symbols  as in  $\prod_n A_n$ in 
\S\ref{S.Reduced}. 
If the language of $A$ has a distinguished constant $0_{\cS}$ for every sort $\cS$, consider   a  
submodel of $\Cb(X,A)$ defined as follows ($\cS$ ranges over all sorts in the language of $A$). 
\begin{multline*}
\textstyle C_0(X,A)=\bigcup_{\cS}\{f\in \cS^{\Cb(X,A)}: \text{ function } x\mapsto d(f(x),0_{\cS}^A)\\
\text{ vanishes at  infinity}\}. 
\end{multline*}
When $X$ is $\omega$ with the discrete topology, these models are isomorphic to $\prod_n A$ and 
$\bigoplus_n A$, respectively. 
In general $\Cb(X,A)$ is a submodel of $\prod_{t\in X} A$. 
In \S\ref{S.Continuous.Field} we make a few remarks on general continuous fields of metric structures.

Consider the quotient structure $\Cb(X,A)/C_0(X,A)$ with the interpretations of predicate and function symbols as defined in \S\ref{S.Reduced}.  
The question whether this quotient is countably saturated is  quite sensitive to the 
choice of $X$ and $A$. Before plunging into the main discussion we record a few relatively straightforward limiting facts. 

\begin{prop} 
\label{P.NS} 
Let  $X$ be a locally compact, non-compact Polish space and let $A$ be a metric structure. 
Write $\cA:=\Cb(X,A)/C_0(X,A)$. 
\begin{enumerate}
\item \label{P.NS.1} If every connected component of $X$ is compact then $\cA$
is countably saturated. 
\item  \label{P.NS.2}  If $X$ is connected and $A$ is discrete then $\cA\cong A$. 
\item \label{P.NS.3} If $X$ is connected and $A$ has a countably  infinite definable discrete subset then $\cA$ is not countably saturated. 
\item \label{P.NS.4} If $X$ is connected and $A$ is a separable, 
infinite-dimensional abelian C*-algebra with infinitely many projections then 
$\cA$ is not countably saturated. 
\end{enumerate}
\end{prop} 

\begin{proof} 
\eqref{P.NS.1} If $X=\bigoplus_n K_n$ and each $K_n$ is compact, then 
with $A_n=\Cb(K_n,A)$ we have 
$\Cb(X,A)/C_0(X,A)\cong\prod_n A_n/\bigoplus_n A_n$ 
and the assertion follows from Theorem~\ref{T.MT.1}. 

\eqref{P.NS.2} If $A$ is discrete and $X$ is connected then $\Cb(X,A)\cong A$,  $C_0(X,A)$ 
is a singleton,  the quotient is isometrically isomorphic to $A$, and  
\eqref{P.NS.3} is an immediate consequence of \eqref{P.NS.2}. 
In a C*-algebra the set of projections is definable by the formula
$\|x^2-x\|+\|x^*-x\|$. This is a consequence of the weak stability of these relations. 
Since two distinct commuting projections are at distance 1, 
this set is discrete. Hence \eqref{P.NS.4} is a special case of \eqref{P.NS.3}. 
\end{proof} 

The following is a limiting example showing why one of the assumptions of 
Theorem~\ref{T2} below is needed. 

\begin{thm} \label{P.NS.5} There exists a one-dimensional, connected space
 $X$ such that $\cA$ is not countably 
saturated for any unital C*-algebra $A$.  
\end{thm} 

\begin{proof} With  $Y=\{0\}\times [0,\infty)$ and $X_n=\{n\}\times [0,\infty)$ for $n\in \omega$
let $X=Y\cup \bigcup_n X_n$ with the subspace topology inherited from $\bbR^2$. 

We give a proof in the case when $A=\bbC$. The general case is a straightforward extension
of the argument by using continuous functional calculus.

Define $a_n\in \Cb(X,A)$ for $n\in \omega$ so that 
$a_n(x)=1$ if $x\in X_n$, $a_n(x)=0$ if $x\in X_j$ for $j\neq n$, 
$a_n(x)=0$ on $\{0\}\times [n+1,\infty)$, and on $\{0\}\times [0,n-1] $ (with $[0,k]=\emptyset$ 
if $k<0$), and linear on intervals $\{0\}\times [n-1,n]$ and $\{0\}\times [n,n+1]$. 

Although $X$ is connected and  $\Cb(X,A)$ is projectionless we have that  $q(a_n)$, for $n\in \omega$, 
are orthogonal   projections in $\Cb(X,A)/C_0(X,A)$.  

\begin{claim} For every projection $p$ in $\Cb(X,A)/C_0(X,A)$ there exists a finite $F\subseteq \omega$
such that $p=\sum_{n\in F} q(a_n)$ or $p=1-\sum_{n\in F} q(a_n)$. 
\end{claim} 

\begin{proof} 
Let $a\in \Cb(X,A)$ be such that 
  $q(a)=p$. 
By replacing $a$ with $(a^*+a)/2$ and truncating its spectrum by using continuous functional 
calculus we may assume $a$ maps 
$X$ into $[0,1]$. 
Fix $\e<1/2$. 
Since $q(a)$ is a projection, the set 
 \[
 Z=\{x\in X: \e<a(x)<1-\e\}
 \]
  is  compact. 
  Hence there exists a compact $K$ 
  so that on every $X_n\setminus K$ and on $Y\setminus K$ all values of $a$ lie outside of $Z$. 
If $G_0=a^{-1}([0,\e])$ and $G_1=a^{-1}([1-\e,1])$ then by the continuity of $a$ 
one of $Y\setminus G_0$ or $Y\setminus G_1$ 
is compact. 

Assume for a moment that  $Y\setminus G_0$ is compact.  
By the continuity of $a$ there exists $k$ such that   $X_n\setminus G_0$ is compact for all $n\geq k$. 
With  $F=\{n: X_n\setminus G_1$ is compact$\}$ letting $\e\to 0$ one easily checks that $p=\sum_{n\in F} q(a_n)$. 

Otherwise, if $Y\setminus G_1$ is compact, a similar argument shows that 
 $F=\{n: X_n\setminus G_0$ is compact$\}$ is  a finite set and $p=1-\sum_{n\in F} q(a_n)$. 
\end{proof}

Consider  type $\bft(x)$   consisting of conditions $\|x^*-x\|=0$, $\|x^2-x\|=0$, $\|x\|=1$, and $x 
q(a_n)=0$ for all $n$. 
Its realization would be a nonzero projection orthogonal to $q(a_n)$ for all $n$. 
By Claim, there is no such projection. 
\end{proof} 

Theorem~\ref{P.NS.5} provides a countably degree-1 saturated abelian C*-algebra such that 
the Boolean algebra of its projections is isomorphic to the subalgebra of $\cP(\omega)$ generated by finite subsets, 
and therefore not countably saturated. 

In \cite[Lemma~2.1]{FaHa:Countable} the following was 
proved for all $a<b$ in $\bbR$ and every   countably degree-1 saturated algebra $C$. 
Every contable type 
such that each of its finite subsets is approximately satisfiable in $C$ by a self-adjoint element with spectrum included in  $[a,b]$ 
is realized in $C$ by a self-adjoint element with spectrum included in $[a,b]$. The following implies that 
this cannot be extended to projections, answering a question raised on p. 53 of \cite{FaHa:Countable}.

\begin{cor}\label{C.P} There are a countably degree-1 saturated C*-algebra $C$ and 
a countable 1-type $\bft$ over $C$ such that every finite subset of $\bft$ is realizable 
by a projection in $C$, but $\bft$ is not realizable by a projection in~$C$. 
\end{cor} 

\begin{proof} Let $C$ be $\Cb(X)/C_0(X)$, with $X$ as in Theorem~\ref{P.NS.5} 
and let $\bft$ be as defined there. 
\end{proof}  

We should also remark that \cite[Question~5.3]{FaHa:Countable}, intended as a test for the question 
answered in Corollary~\ref{C.P}, trivially has a negative answer in a projectionless algebra such as $C([0,1)^*)$. 

The Calkin algebra also fails to be quantifier-free countably saturated (\cite[\S 4]{FaHa:Countable}), 
although all coronas of separable C*-algebras are countably degree-1 saturated
(\cite[Theorem~1]{FaHa:Countable}, see also \cite{Voi:Countable} and \cite{EaVi:Degrees}). 
No example of an algebra which is quantifier-free countably saturated but not
countably saturated is presently known.

Recall that the 
topological boundary
of a subset $K$ of a topological space $X$ is $\partial K:=
\overline K\cap \overline{X\setminus K}$.

 \begin{thm} \label{T2} 
 Assume $X$ is a locally compact Polish  space 
 which  can be written as an increasing union of 
compact subspaces, $X=\bigcup_n K_n$, such that 
\[
\sup_n|\partial K_n|<\infty.
\] 
Then for any metric model $A$ 
such that each domain of $A$ is compact and 
locally connected 
the quotient $\Cb(X,A)/C_0(X,A)$ 
is countably saturated. 
 \end{thm} 

Proof of this theorem is given below in  \S\ref{P.T2}. 
Theorem~\ref{P.NS.5} 
shows that the assumption 
on $\partial K_n$ in Theorem~\ref{T2} cannot be relaxed. 
The assumption  that the compact sets $K_n$ can be chosen to satisfy 
$\sup_n |\partial K_n|<\infty$ implies that for every $\e>0$ and every sort in $\cL$ 
whose interpretation is a compact subset of $A$ there exists $N$ 
such that every  $\e$-discrete subset of this sort in~$A^{\partial K_n}$ has cardinality $\leq N$.

 \begin{proof}[Proof of Theorem~\ref{T.I.1}]
  Apply Theorem~\ref{T2} to $\Cb(\bbH,\bbC)/C_0(\bbH,\bbC)$. 
   \end{proof}

\subsection{Relevant patterns} The notion of relevant pattern from \S\ref{S.Deconstructing} 
is modified to  present context in the natural way. 
For a tuple of formulas $\phi_0(\bar x), \dots, \phi_{k-1}(\bar x)$ in prenex normal form and 
 $\bar a$ in $\Cb(X,A)$ of the 
appropriate sort let  $P_t$ denote the pattern of $\phi_0(\bar x), \dots, \phi_{k-1}(\bar x)$ 
 and  $a_t$ in $A$ (where $a_t$ denotes the value of $a$ at $t$). 
An  $n$-pattern $P$  is \emph{relevant} (for $\tilde \phi_j(\bar a)$, $0\leq j<k$, 
in $\Cb(X,A)$) 
if   closure of the set $\{t: P=P_t\}$ is not included in $X$. 
Equivalently, an $n$-pattern 
$P$ is relevant (for $\tilde \phi_j(\bar a)$, $0\leq j<k$, 
in $\Cb(X,A)$) if for every $k$ there exists $t\in X\setminus K_k$ such that 
$P=P_t(m,n)$. 

The following is analogue of Lemma~\ref{L2}  for $\Cb(X,A)/C_0(X,A)$. 
Its proof is very similar to the proof of Lemma~\ref{L2}. 

\begin{lemma}  For every formula $\phi(\bar x)$  
 and  $\e>0$ there is $m$ such that 
 for every quotient  $C_b(X,A)/C_0(X,A)$
 and tuple $\bar a\in C_b(X,A)$ of the appropriate sort 
 the value of  $\phi(q(\bar a))$ in $C_b(X,A)/C_0(X,A)$ is determined up to $\e$ 
 by the set of all relevant $m$-patterns for $\phi(\bar x)$ and $\bar a$ in $\Cb(X,A)$. \qed
\end{lemma}

\subsection{Proof of Theorem~\ref{T2}}\label{P.T2} 
We roughly follow the proof of Theorem~\ref{T.MT.1}. 

 Let $\bft=\{\phi_i(\bar y)=r_i: i\in \omega\}$ be a type over  
  $\Cb(X,A)/C_0(X,A)$ 
consistent with its theory. By Lemma~\ref{L.PNF} we may assume that all $\phi_i$ are given in prenex normal form.  
Lift parameters of $\phi_i(\bar y)$ to $\Cb(X,A)$ to obtain formulas $\tilde \phi_i(\bar y)$. 
For every $m$ choose   $\bar a(m)\in \Cb(X,A)$ so that 
\[
|\phi_i(q(\bar a(m)))^{\Cb(X,A)/C_0(X,A)}-r_i|\leq 2^{-m}
\] 
for all $i\leq m$. 

Let $R(m,n)$ denote the set of relevant $n$-patterns for $\tilde \phi_j(\bar a(m))$, $0\leq j<m$ 
in $\prod_n A_n$.

Main difference between the present proof and the proof of Theorem~\ref{T.MT.1}  is that when 
`cutting and pasting' partial realizations of $\bft$ we need to assure continuity  of the resulting element $\bar a$. 

For every $k$ consider the set $A^{\partial K_k}$ with respect to the sup metric $d$. 
These sets are compact,  and since cardinalities of $\partial K_k$, for $k\in \omega$,   are uniformly  bounded
for every $\e>0$ there exists $N_\e<\infty$ such that no $A^{\partial K_k}$ has an $\e$-discrete subset
of cardinality $>N_\e$. For the simplicity of notation for every $k$ 
we introduce a pseudometric on $C_b(X,A)$ by 
\[
d_k(a,b)=d(a\rs\partial K_k, b\rs \partial K_k). 
\]
Recursively choose a sequence  of  infinite sets
$\omega=Y_0\supseteq Y_1\supseteq\dots$, 
  so that for every $n$ and all $m_1$ and $m_2$ in $Y_n$ we have the following.  
\begin{enumerate}
\item\label{2.l1}  $R(m_1,n)=R(m_2,n)$. 
\item \label{2.l2} $\{k: d_k(a(m_1), a(m_2))\leq 1/n\}$ is infinite. 
 \pushcounter
  \end{enumerate}
Assume that $Y_{n-1}$ has been chosen to satisfy the conditions. 
Consider the set of unordered pairs of elements of $Y_{n-1}$, 
\[
[Y_{n-1}]^2=\{\{m,m'\}\subseteq Y_i: m\neq m'\}
\]
and define partition $c\colon [Y_{n-1}]^2\to \{0,1\}$ by 
$c(\{m_1,m_2\})=0$ if \eqref{2.l2} holds and $c(\{m_1,m_2\})=1$ otherwise.

Assume for a moment that there exists an infinite $Z\subseteq Y_{n-1}$ 
such that $c(\{m,m'\})=1$ for all $\{m,m'\}\in [Z]^2$.  
With $N_{1/n}$ for which  there are no $1/n$-discrete subsets of $C_b(X,A)$ of 
cardinality $N_{1/n}$ for any $d_k$, consider the least $N_{1/n}$ elements of $Z$, listed as  
$m_i$, for $1\leq i\leq N_{1/n}$.  
For all $i<j$ and every large enough $k$ we have $d_k(a(m_i),a(m_j))>1/n$. 
Therefore $a(m_i)$, for $1\leq i\leq N_{1/n}$, form a $1/n$-discrete set for some $k$, 
contradicting the choice of $N_{1/n}$. 

By Ramsey's theorem there exists an infinite  $Y\subseteq Y_{n-1}$ such that $c(\{m_1,m_2\})=0$ for all
$\{m_1,m_2\}\in [Y]^2$.
Therefore $Y_{n-1}$ has an infinite subset for which  \eqref{2.l2} holds.  Since there are only finitely 
many relevant patterns we can find an infinite 
subset $Y_n$ of $Y$ satisfying \eqref{2.l1} and proceed.

This describes the recursive construction. Now we follow  construction 
from the proof of Theorem~\ref{T.MT.1} and 
recursively choose increasing sequences  $k(i)$ and $m(i)\in Y_i$ for $i\in \omega$  
  so that for every $n$  we have the following.  
\begin{enumerate}
\popcounter
\item \label{2.l4} The set of all patterns of 
  $\tilde\phi_0(\bar a(m(n))_j), \dots, \tilde\phi_{n-1}(\bar a(m(n))_j)$  occurring in $A_t$
  for some $t\in K_{k(n+1)}\setminus K_{k(n)}$ is equal to $R(m(n),n)$. 
\item \label{2.l3} $d_{k(n+1)}(a(m(n)), a(m({n+1})))\leq 1/(n+1)$. 
\pushcounter
\end{enumerate}
Assume $m(i)$ and $k(i)$ for $i\leq n$  were chosen so that \eqref{2.l4} and \eqref{2.l3} hold and moreover all patterns of 
$\tilde\phi_0(\bar a(m(n))_j), \dots, \tilde\phi_{n-1}(\bar a(m(n))_j)$  
occurring in $A_t$ for $t\notin K_{k(n)}$ are relevant. 
 Pick the least $m(n+1)\in Y_{n+1}$ greater than $m(n)$.
Since there are infinitely many $k$ such that $d_k(a(m(n)), a(m({n+1})))\leq 1/(n+1)$, 
we can choose large enough $k(n+1)$ so that both \eqref{2.l4} and \eqref{2.l3} are satisfied
and 
all patterns of 
$\tilde\phi_0(\bar a(m(n+1))_j), \dots, \tilde\phi_{n-1}(\bar a(m(n+1))_j)$  
occurring in $A_t$ for $t\notin K_{k(n+1)}$ are relevant. 

Once  these objects are chosen one is tempted to  define $\bar a\in \Cb(X,A)$ via  
\[
\bar a_t=\bar a(m(i))_t,  \text{ if } t\in K_{k(i+1)}\setminus K_{k(i)}. 
\]
However, for this function the map $t\mapsto d(\bar a_t,0^{\Cb(X,A)})$ 
is not necessarily continuous and therefore $\bar a$ may not be in  $\Cb(X,A)$. 
 Nevertheless, $a_{m(i)}$ and $a_{m(i+1)}$ differ by at most $1/(i+1)$ on $\partial K_{k(i+1)}$ 
 and we proceed as follows. 
 
 Fix $n$. Let $\e_n$ be such that $d(\bar x, \bar y)<\e_n$ implies 
 $\max_{i\leq n}| \phi_i(\bar x)-\phi_i(\bar y)|<1/n$. 
  By using finiteness of $\partial K_{k(n+1)}$ fix an open neighbourhood $U_n$
   of $\partial K_{k(n+1)}$ such that both $\bar a_{m(n)}$ and $\bar a_{m(n+1)}$ 
vary by at most $\e_n$ on $U_n$. 
Also assure that $U_n\subseteq K_{k(n+2)}$ and $U_m\cap U_n=\emptyset$ if $m\neq n$.

By Tietze's extension theorem recursively 
choose $h_n\colon X\to [0,1]$ such that for all $m$ and $n$ we have the following
(here $1_Z$ denotes the characteristic function of $Z\subseteq X$)
\begin{enumerate}
\popcounter
\item $0\leq h_n\leq 1$, 
\item  $h_{n-1}+h_n +h_{n+1}\geq 1_{K(k(n)}-1_{K_{k(n-1)}}$, 
\item $h_n+h_{n+1}\geq 1_{\partial K_{k(n)}}$, 
\item $\sum_n h_n=1_X$, 
\item $\supp h_m\cap \supp h_n=\emptyset$ if $|m-n|\geq 2$,  
\item $\supp h_n\cap \supp h_{n+1}\subseteq U_n$. 
\end{enumerate}
  Then 
\[
\textstyle \bar a_t=
\sum_n
h_n(t)\bar a(m(n))_t
\]
 is an element of $\Cb(X,A)$ which agrees 
 with $a_{m(n)}$ on $K_{k(n+1)}\setminus K_{k(n)}$ 
up to $\e_n$ 
 for all $n$. Moreover, the set of   relevant $n$-patterns of $\bar a$ 
is equal to $R(m(i), n)$ for infinitely many $i$. As before, this implies that $\bar a$ realizes type 
$\bft$. 
 
 This concludes the proof of Theorem~\ref{T2}. 
 
 \subsection{Continuous fields of models} 
In addition to models $\Cb(X,A)$ and $C_0(X,A)$, one can consider  
the following submodel of the former. 
\begin{multline*}
C_c(X,A)=\{f\colon X\to A :  f\text{ is continuous, its range is included in}\\
\text{  a single domain of $A$,
and it has compact closure.}\}. 
\end{multline*}
This smaller model has the  property that $C_c(X,A)\cong \Cb(X,A)$ 
if $X$ is compact and $C_c( X, A)\cong \Cb(\beta X,A)$ otherwise.

A possible definition of
 a \emph{continuous field of models} $A_t$, for $t\in X$, can be obtained  by fixing a 
 sufficiently saturated model (commonly called `monster model')  $M$,  
 requiring all fibers $A_t$ to be submodels of $M$ and considering the model 
 consisting of all continuous $a\colon Y\to M$ such that $a_t\in A_t$ for all $t$. 

However, 
one motivating example is given by  continuous fields of C*-algebras
and in this case one does not expect  all fibers to be  elementarily equivalent. 
It is not difficult to modify the definition   in \cite[IV.6.1]{Black:Operator} to the general context of 
continuous fields of metric structures of an arbitrary signature. At the moment we do not have an application for this notion, but as model-theoretic methods are gaining prominence in the 
theory of operator algebras this situation is likely to change. 
On a related note, sheaves of metric structures were defined
in \cite{lopes2013reduced}.

 \subsection{Reduced products over ideals other than the Fr\'echet ideal} 
\label{S.Generalizations}
Reduced products over a nontrivial proper ideal are countably saturated in both  extremal cases: 
when the ideal is maximal (i.e., when the quotient is  an ultrapower) and when we have the Fr\'echet
ideal (Theorem~\ref{T.MT.1}). We consider some of the intermediate cases. 

Assume $A_n$, for $n\in \omega$, are $\cL$-structures as in 
Theorem~\ref{T.MT.1} and that the language has a distinguished symbol $0^{\cS}$ for every $\cL$-sort $\cS$.   
For an ideal~$\cI$ on $\omega$ define (here $\cS$ ranges over all sots in $\cL$) 
\[
\textstyle \bigoplus_{\cI} A_n
:=\bigcup_{\cS} \{\bar x\in \cS^{\prod_n A_n}: (\forall \e>0)(\exists X\in \cI)  
\sup_{n\notin X}d(x_n,0_{\cS})\leq \e\}. 
 \]
 Therefore if $\cI$ is the Fr\'echet ideal then $\bigoplus_{\cI}$ is the standard direct sum. 
 We shall show that for many ideals~$\cI$  
   quotient structure $\prod_n A_n/\bigoplus_{\cI} A_n$
is countably saturated, analogously to the situation in the first-order logic. 
Following   \cite[Definition~6.5]{Fa:CH} we say that 
an ideal~$\cI$ on $\omega$ is \emph{layered} if there is $f\colon \cP(\omega)\to [0,\infty]$
such that
\begin{enumerate}
\item $A\subseteq B$ implies $f(A)\leq f(B)$,
\item  $\cI=\{A: f(A)<\infty\}$,
\item  $f(A)=\infty$ implies $f(A)=\sup_{B\subseteq A} f(B)$.
\end{enumerate}
The following extends \cite[Lemma~6.7]{Fa:CH}. 

 \begin{thm} \label{T3} 
 Every reduced product $\prod_n A_n/\bigoplus_{\cI} A_n$ 
 over a  layered ideal  
 is countably 
 saturated. 
 \end{thm}

\begin{proof} We follow  a similar route as in the  proof of 
Theorem~\ref{T.MT.1}. Assume $\bft=\{\phi_i(\bar x)=r_i: i\in \omega\}$ is a type and $\bar a(m)$ is such 
that $|\phi_i(q(\bar a(m)))^{\cA}-r_i|\leq 1/m$ 
for all $i\leq m$. By Lemma~\ref{L.PNF} we may assume that all $\phi_i$ are given in prenex normal form. 

Lift parameters of $\phi_i$, for $i\in \omega$, 
 to $\prod_n A_n$ to obtain formulas $\tilde\phi_i$, for $i\in \omega$.
Let $P_i(m,n)$ be the $n$-pattern 
of $\tilde\phi_j(\bar a(m))$, for $0\leq j<n$, in $A_i$.  
An $n$-pattern $P$ of 
$\tilde\phi_j(\bar a(m))$, for $0\leq j<n$, 
  is  \emph{$\cI$-relevant}  (for $\tilde\phi_0, \dots, \tilde\phi_{m-1}, \bar a(m)$) 
 if the set  
\[
\{i: P=P_i(m,n)\}
\]
 is $\cI$-positive. 
Let 
\begin{align*}
R(m,n)=\{P: & \text{ $P$ is an $\cI$-relevant pattern}\\
& \textstyle\text{ for }\tilde\phi_0, \dots, \tilde\phi_{m-1}, \bar a(m)\text{ in }\prod_n A_n\}. 
\end{align*}
Thus patterns relevant in the sense of \S\ref{S.Deconstructing} 
are $\Fin$-relevant. 

The analogue of Lemma~\ref{L2}, that for every $\phi$ and every 
$\e>0$ there exists $n$ such that the value of $\phi$ in $\prod_n A_n/\bigoplus_{\cI} A_n$ 
depends only on relevant $n$-patterns, is easily checked. 

Since $\cI$ includes the Fr\'echet ideal, we can choose sets $\emptyset=
Y_0\subseteq Y_1\subseteq $ in 
$\cI$, sets $\omega=X_0\supseteq X_1\supseteq\dots$
  and a sequence $m(i)$, for $i\in \omega$, 
such that $\bigcup_n Y_n=\omega$
and moreover we have the following for all $n$.  
% every $\bar a(m(i))$ 
 \begin{enumerate}
\item  $R(m,n)=R(m(n),n)$ for
all $m\in X_n$. 
\item The set of $n$-patterns of 
  $\tilde\phi_0(\bar a(m)_j), \dots, \tilde\phi_{n-1}(\bar a(m)_j)$  in $A_j$
  for $j\in Y_m\setminus Y_{m-1}$ is equal to $R(m,n)$. 
  \end{enumerate}
Recursive construction of these sequences is analogous to one in the proof of Theorem~\ref{T.MT.1}. As before, one concatenates $\bar a(m(i))$ into $\bar a$ and checks that $q(\bar a)$
realizes $\bft$. It is important to note that for every $n$ 
all $n$-patterns of $\tilde\phi_0, \dots, \tilde\phi_{m-1}, \bar a$ in $A_j$ for  $j\in \omega\setminus Y_n$ 
coincide. The analogue of Lemma~\ref{L2} stated above then implies that $\bar a$ realizes $\bft$. 
\end{proof} 

All $F_\sigma$ ideals are layered (essentially \cite{JustKr}), but there are Borel layered ideals of 
an arbitrarily high complexity (\cite[Proposition~6.6]{Fa:CH}). 
Examples of  non-layered ideals are the ideal of asymptotic density zero sets, 
$\cZ_0=\{X\subseteq \omega: \lim_n |X\cap n|/n=0\}$ and the ideal 
of logarithmic density zero sets, 
$\cZ_{\log}=\{X\subseteq \omega: \lim_n (\sum_{k\in X\cap n} 1/k)/\log k=0\}$. 
 Their quotient Boolean algebras are not 
countably quantifier-free saturated since any strictly decreasing sequence of positive 
sets whose upper 
densities converge to zero has no nonzero lower bound.  
However, a fairly technical  construction of an isomorphism 
between $\cP(\omega)/\cZ_0$ and $\cP(\omega)/\cZ_{\log}$ in  \cite{JustKr} 
(cf. also \cite[\S 5]{Fa:CH}) using the Continuum Hypothesis
resembles a back-and-forth construction of an isomorphism between elementarily equivalent 
countably saturated structures. 
Logic of metric structures puts this apparently technical result into the correct context.

A map $\phi\colon \cP(I)\to [0,\infty)$ is a  \emph{submeasure} if it is subadditive, 
monotonic and satisfies $\phi(\emptyset)=0$. 
If $\omega=\sqcup_n I_n$ is a partition of $\omega$ into finite intervals and 
$\phi_n$ is a submeasure on $I_n$ for every $n$, then the ideal 
\[
\cZ_\phi=\{A\subseteq \omega: \limsup_n \phi_n(A\cap I_n)=0\}
\]
is a \emph{generalized density ideal} (see \cite[\S 2.10]{Fa:CH}). 
Ideals $\cZ_0$, $\cZ_{\log}$, and all of the so-called \emph{EU-ideals} (\cite{JustKr}) are 
(generalized) density ideals (\cite[Theorem~1.13.3]{Fa:AQ}). 

On the other hand,  generalized density ideals
 are a special case of ideals of the form 
\[
\Exh(\phi)=\{A\subseteq \omega: \limsup_n \phi(A\setminus n)=0\}
\]
where $\phi$ is a lower semicontinuous submeasure on $\omega$
(take $\phi(A)=\sup_n \phi_n(A\cap I_n)$). 
All such ideals are $F_{\sigma\delta}$ P-ideals and 
by Solecki's theorem (\cite{Sol:Analytic}), every analytic P-ideal 
is of this form   for some $\phi$. 

For such ideal quotient Boolean algebra $\cP(\omega)/\Exh(\phi)$ is equip\-ped with the complete metric
($q$ denotes the quotient map)
\[
d_\phi(q(A), q(B))=\liminf_n \phi((A\Delta B)\setminus n). 
\]
See \cite[Lemma~1.3.3]{Fa:AQ} for a proof.

\begin{prop} \label{P.JK} If $\Exh(\phi)$ is a generalized density ideal then the 
quotient $\cP(\omega)/\Exh(\phi)$ with respect to $d_\phi$ 
 is a countably saturated metric Boolean algebra. 
 \end{prop}

\begin{proof} Let $I_n$ and $\phi_n$ be as in the definition of generalized density ideals. 
Letting $A_n$ be the metric Boolean algebra $(\cP(I_n), d_n)$ where 
$d_n(X,Y)=\phi_n(X\Delta Y)$. 
We have that $\cP(\omega)/\Exh(\phi)$ is isomorphic to $\prod_n A_n/\bigoplus_n A_n$
by \cite[Lemma~5.1]{Fa:CH}. 
By Theorem~\ref{T.MT.1}, this is a countably saturated metric structure. 
\end{proof} 

%%%%%%%%%automorphisms!!!

By Proposition~\ref{P.JK} the main result  of \cite{JustKr} is equivalent 
to the assertion that the quotients over $\cZ_0$ and $\cZ_{\log}$ equipped with the 
canonical metric are elementarily equivalent
in logic of metric structures. The latter assertion is more elementary and absolute between transitive models of ZFC. 
We note that it cannot be proved in ZFC that the quotients over $\cZ_0$ and $\cZ_{\log}$ are  isomorphic (\cite{Just:Repercussions}, 
see also \cite{Fa:AQ} and \cite{Fa:Rigidity}). 

 Proposition~\ref{P.JK} and the ensuing  discussion 
beg several questions. Can one describe  
theories of  quotients over analytic P-ideals in the logic of metric structures? 
How does the theory depend on the choice of the metric? 
Are those quotients countably saturated whenever the ideal includes
the Fr\'echet ideal? A positive answer would imply that
the Continuum Hypothesis implies that a quotient Boolean algebra over a nontrivial 
analytic P-ideal has $2^{\aleph_1}$ automorphisms, complementing the 
main result of \cite[\S 3]{Fa:AQ} and partially answering a question of Juris Stepr\=ans. 
A related problem is to extend the Feferman--Vaught theorem (\cite{feferman1959first}) 
to reduced products  of metric structures. The discretization method from \S\ref{S.Deconstructing} 
 should be relevant to this problem. 
Question of the existence of a nontrivial Borel ideal with a rigid quotient Boolean 
algebra will be treated in an upcoming paper.

 \section{Automorphisms} 
\label{S.Automorphisms}
 
 Proof of the following theorem is due to Bradd Hart and it is included with 
 his kind permission. 
 
 \begin{thm} \label{T.Aut} 
 Assume $A$ is a $\kappa$-saturated structure of density character $\kappa$. 
 Then $A$ has $2^\kappa$ automorphisms. 
 \end{thm} 
 
 \begin{proof} Enumerate  a dense subset of $A$ as  $a_\gamma$ for $\gamma<\kappa$. 
 We follow von Neumann's convention and write $2$ for $\{0,1\}$,  
 consider $2^{<\kappa}=\bigcup_{\gamma<\kappa} 2^\gamma$ and write $\len(s)=\gamma$
 if $s\in 2^\gamma$. 

 We construct families $f_s$ and $A_s$ for $s\in 2^{<\kappa}$ such that the following holds for all $s$. 
 \begin{enumerate}
\item $A_s$ is an elementary submodel of $A$ of density character $<\kappa$ including $\{a_\gamma: \gamma<\len(s)\}$. 
 \item $f_s\in \Aut(A_s)$. 
 \item If $s\sqsubseteq t$  then $A_s\prec A_t$ and 
 $f_t\rs A_s=f_s$. 
 \item\label{T.Aut.3} $A_{s^\frown 0}=A_{s^\frown 1}$ but $f_{s^\frown 0}\neq f_{s^\frown 1}$. 
 \end{enumerate}
 The only nontrivial step in the recursive construction is to assure \eqref{T.Aut.3}.  
 Fix $\gamma$ such that  $A_s$ and $f_s$ for all $s\in 2^\gamma$ 
 have been chosen and satisfy the above requirements. Fix one of these $s$.   
 Choose $a_\xi$ with the least index $\xi$ which is not in $A_s$ and let $\e=\dist(a_\xi, A_s)$.  
 If $\bft(x)$ is the type of $a_\xi$ over $A_s$, 
 then let $\bfs(x,y)$ be the 2-type $\bft(x)\cup \bft(y)\cup \{d(x,y)\geq \e\}$. 
Since $A_s$ 
is an elementary submodel of $A$, every finite subtype of $\bft(x)$ is realized in $A_\gamma$ and
therefore $\bfs(x,y)$ is consistent. By the saturation of $A$ it is realized by a pair 
of elements $b_1, b_2$ in $A$. Again by saturation (and the smallness of $A_s$) we can extend $f_s$ to an automorphism of $A$ that sends $b_1$ to $b_2$. Call this automorphism $g_1$. 
Let $g_0$ be an automorphism of $A$ that extends $f_s$ and sends $b_1$ to itself. 
By a L\"owenheim--Skolem argument we can now choose an elementary 
submodel $A_{s^\frown 0}$ of $A$ which includes $A_s\cup \{a_\xi, b_1\}$ 
and is closed under both $g_0$ and $g_1$. Let $A_{s^\frown 1}=A_{s^\frown 1}$.  
Then $g_0\rs A_{s^\frown 0}\neq g_1\rs A_{s^\frown 0}$ 
and  
$f_{s^\frown i}=g_i\rs A_{s^\frown i}$ for $i\in \{0,1\}$ are as required.  

This describes the recursive construction. For every $s\in 2^\kappa$ we have that 
$A=\bigcup_{\gamma<\kappa} A_{s\rs \gamma}$ and $f_s:=\bigcup_{\gamma<\kappa} f_{s\rs \kappa}$ for $s\in 2^\kappa$ 
are distinct automorphisms of $A$.
\end{proof}  
 
In the above proof  it may be possible to assure that $A_s=A_t$ whenever $\len(s)=\len(t)$. 
This is not completely obvious since the model $A_{s^\frown 0}=A_{s^\frown 1}$ defined in the course 
of the proof depends on automorphisms $g_0$ and $g_1$, and therefore on $s$.

 \begin{proof}[Proof of Theorem~\ref{T.I.1}]
By Theorem~\ref{T.I.1}  C*-algebra $C(\bbH^*)$ is countably saturated. 
Therefore by Theorem~\ref{T.Aut} the Continuum Hypothesis implies that it has $2^{\aleph_1}$ automorphisms. 
By the Gelfand--Naimark duality, each of these automorphisms corresponds to a distinct
autohomeomorphism of $\bbH^*$. 
\end{proof}

\section{The asymptotic sequence algebra}\label{S.asymptotic}
If $A_n$, for $n\in \omega$ is a sequence of C*-algebras then the reduced product $\prod_n A_n/\bigoplus_n A_n$ is 
the \emph{asymptotic sequence algebra}. If $A_n=A$ for all $n$ then we write $\ell_\infty(A)/c_0(A)$ for 
$\prod_n A/\bigoplus_n A$. Also, algebra $A$ is identified with its diagonal 
image in $\ell_\infty(A)/c_0(A)$ and one considers the relative commutant
\[
A'\cap \ell_\infty(A)/c_0(A)=\{b\in \ell_\infty(A)/c_0(A): a  b=b a\text{ for all } a\in A\}. 
\]
This is the \emph{central sequence algebra}. 
The following corollary provides an explanation of why the asymptotic sequence C*- algebras
  and the central sequence C*-algebras are almost as useful for the analysis of 
  separable C*-algebras as the ultrapowers and the corresponding relative commutants.

\begin{cor} \label{P1} If $A$ is a separable
 C*-algebra then the asymptotic sequence algebra $\ell_\infty(A)/c_0(A)$ 
 is countably saturated and   
  the corresponding central sequence algebra $A'\cap \ell_\infty(A)/c_0(A)$ 
 is  countably quantifier-free saturated. 

The  Continuum Hypothesis implies that each of these algebras has $2^{\aleph_1}$ automorphisms
and that $\ell_\infty(A)/c_0(A)$ is isomorphic to its ultrapower associated with a nonprincipal ultrafilter on~$\omega$.  
\end{cor}

\begin{proof}
Assume $A$ is a nontrivial separable C*-algebra. 
The asymptotic sequence algebra $\ell_\infty(A)/c_0(A)$ is countably 
saturated by Theorem~\ref{T.MT.1}. 
Countable quantifier-free saturation of the central sequence algebra $A'\cap \ell_\infty(A)/c_0(A)$
follows by \cite[Lemma~2.4]{FaHa:Countable}. 
By Theorem~\ref{T.Aut} Continuum Hypothesis implies that the asymptotic sequence  algebra has $2^{\aleph_1}$ 
automorphisms.   
A diagonal argument shows that for a separable $B\subseteq \ell_\infty(A)/c_0(A)$ the relative commutant $B'\cap \ell_\infty(A)/c_0(A)$
is nonseparable. Then  the proof of Theorem~\ref{T.Aut}
shows that one can construct $2^{\aleph_1}$ of its automorphisms that 
pointwise fix $A$ and differ on its relative commutant, and therefore the central sequence algebra has $2^{\aleph_1}$ automorphisms. 

Finally, if $\cU$ is a nonprincipal ultrafilter then the ultrapower of $\ell_\infty(A)/c_0(A)$
is elementarily equivalent to itself and of cardinality $2^{\aleph_0}$. 
Since elementarily equivalent saturated structures of the same density character are 
isomorphic, this concludes the proof. 
  \end{proof} 
  
  We record a (very likely well-known) corollary. 
  
  \begin{prop} For unital separable C*-algebras $A$ and  $B$  and every 
  nonprincipal ultrafilter $\cU$ on $\omega$ the following are equivalent. 
  \begin{enumerate}
  \item there exists
  a unital *-homomorphisms of $B$ into the ultrapower~$A^{\cU}$. 
  \item There exists a unital *-homomorphism of $B$ into $\ell_\infty(A)/c_0(A)$. 
  \end{enumerate}
  \end{prop} 
 
 \begin{proof} Recall that for a countably saturated structure $C$
 and a separable structure   $B$ of the same language we have that $B$ embeds  into $C$ 
 if and only if the \emph{existential theory} of $B$, 
 \[
 \ThE(B)=\{\inf_{\bar x}\phi(\bar x): \inf_{\bar x}\phi(\bar x)^B=0
 \text{ and $\phi$ is quantifier-free}\}
 \]
 is included in $\ThE(C)$ (\cite{FaHaSh:Model2}). 
Both $A^{\cU}$ and $\ell_\infty(A)/c_0(A)$ are countably saturated. 
 Since $A$ and $A^{\cU}$ have the same theory by \L os's theorem
 and $A$ and 
  $\ell_\infty(A)/c_0(A)$ have the same existential theory  
by \cite{lopes2013reduced} the conclusion follows. 
\end{proof}

  A tentative  definition of a trivial automorphism of a corona of a separable C*-algebra
   was given 
  in \cite{CoFa:Automorphisms}. Every inner automorphism is trivial and there 
  can be at most $2^{\aleph_0}$ trivial automorphisms; this is all information 
  that we need for the following corollary. 

  \begin{cor} The assertion that all automorphisms of  the algebra
  $\prod_n M_n(\bbC)/\bigoplus_n M_n(\bbC)$
  are trivial is independent from ZFC. 
  \end{cor} 
  
  \begin{proof} 
  Relative consistency  of this assertion  was proved in  \cite{Gha:FDD}.  
  Continuum Hypothesis implies  that the algebra has $2^{\aleph_1}$ automorphisms by Theorem~\ref{T.MT.1} and Theorem~\ref{T.Aut}. 
  \end{proof}

\section{The influence of forcing axioms} 
\label{S.W} 
\label{S.Rep}
In this section, if $f$ is a function and $X$ is a subset of its domain we write
\[
f[X]=\{f(x): x\in X\}. . 
\]
Let $X$ be a topological space. 
 Let $\cF_X$ and $\cK_X$ denote the lattice of closed subsets of $X$ and 
its ideal of compact sets, respectively. For $F\in \cF_X$ let 
\[
F^*=\overline F\setminus X, 
\]
where closure is taken in $\beta X$. 
Fix a homeomorphism $\Phi\colon X^*\to Y^*$.
If there is a function
$\Psi\colon \cF_X\to \cF_Y$ such that $\Phi[F^*]=\Psi(F)^*$ for all $F\in \cF_X$ then we  
 say that $\Psi$ is a \emph{representation} of $\Phi$.

This is a very weak assumption since  $\Phi$ is just an arbitrary function with this property. 
Nevertheless, 
Continuum Hypothesis implies that there exists a homeomorphism between 
Stone--\v{C}ech remainders of locally compact Polish spaces with no representation 
(see \S\ref{S.Norep}). 

  If $\Phi$ is trivial and $f\colon \beta X\to \beta Y$ is its continuous extension then $\Psi(F)=f(F)$
  defines a  representation of $\Psi$.  
Since points of $\beta X$ are maximal filters  of closed subsets of $X$, $\Psi$ is uniquely determined 
by its representation. 

We shall show (Theorem~\ref{T1}) that PFA implies every homeomorphism $\Phi$ between Stone--\v{C}ech remainders of locally compact Polish spaces such that both $\Phi$ and its inverse have a representation is trivial. 

In the remainder of this section 
we work in ZFC and prove the following. 

\begin{lemma}\label{L-1} Every autohomeomorphism of $\bbH^*$ has a representation. 
\end{lemma}

In the following lemma we say that a closed subset $a$ of $\bbH$ is \emph{nontrivial} 
if neither $a$ nor $\bbH\setminus a$ is relatively compact.

\begin{lemma} \label{L4} Assume $F$ is a closed subset of $\bbH^*$. 
\begin{enumerate}
\item $F$ includes $a^*$ for some
 nontrivial closed $a\subseteq \bbH$ if and only if $\bbH^*\setminus F$ is 
disconnected. 
\item $F=a^*$ for a nontrivial $a\subseteq\bbH$ if and only if $\bbH^*\setminus W$ is 
disconnected for every nonempty relatively open $W\subseteq F$. 
\end{enumerate}
\end{lemma} 

\begin{proof} Let us write $X=\bbH$.
(1) If $F\subseteq X^*$ is a nontrivial closed set then its complement is a union of infinitely 
many nonempty disjoint open sets. Pick a sequence $V_n$, for $n\in \omega$, 
of these sets such that $\lim_n \min V_n=1$. Now let $W_1=\bigcup_n V_{2n}$ and 
$W_2=X\setminus (F\cup W_1)$. Then $W_2$ is open, being a union of open intervals, 
and neither  $W_1$ nor $W_2$ is included in a compact subset of $\bbH$. 
Hence $X^*\setminus F$ is partitioned into two disjoint open sets corresponding to $W_1$ and $W_2$.

Now we  show the converse
implication. 
Let $U$ and $V$ be disjoint open subsets of $X^*$
such that $U\cup V=X^*\setminus F$. 
If $\tilde U$ and $\tilde V$ are open subsets of $\beta X$ such that 
$U=\tilde U\cap X^*$ and $V=\tilde V\cap X^*$ then $\tilde U\cap\tilde V$ is included in a 
compact set, and we can assume it is empty. Since $X$ is connected 
 $a=X\setminus (\tilde U\cup \tilde V)$ is nonempty and moreover it is not 
compact. Since both $U$ and $V$ are nonempty the complement of $a$
is not included in a compact set.  
Therefore $a$ is a nontrivial  closed  set such that $a^*\subseteq F$. 

(2) By (1), only the converse implication requires a proof. 
Let $a=X\setminus (\tilde U\cap \tilde V)$
 be the nontrivial closed subset of $X$ constructed in (1) such that $a^*\subseteq F$. 
We claim that $a^*=F$. Assume, for the sake of obtaining a contradiction, 
 that $W=F\setminus a^*$ is nonempty. 
 By the assumption, $X^*\setminus W$ is disconnected. By compactness 
 let $G\subseteq W$ be a closed set 
 such that $X^*\setminus G$ is disconnected. By (1) fix a nontrivial closed  $b\subseteq X$
 such that $b^*\subseteq G$. Then $a^*\cap b^*$ is empty and therefore $a\cap b$ is compact.  
 Therefore $b$ has noncompact intersection with  
 $\tilde U\cup \tilde V$ used in (1), and therefore $b^*$ has a nonempty intersection with $U\cup V$. 
   But this is absurd since $U\cap  V $ is disjoint from $F$. 
 \end{proof} 

\begin{proof}[Proof of Lemma~\ref{L-1}]
Let $\Phi\colon \bbH^*\to \bbH^*$ be a homeomorphism. 
Fix a closed $a\subseteq \bbH$. We need to find closed $b\subseteq\bbH$ such that 
$\Phi(a^*)=b^*$. If $a^*=\bbH^*$ or $a^*=\emptyset$ then this is easy. 
Otherwise, if both $a^*$ and $\bbH^*\setminus a^*$ are nonempty, then  
Lemma~\ref{L4} implies that $\bbH^*\setminus W$ is disconnected 
for every nonempty relatively open $W\subseteq a^*$. 
Therefore $\bbH^*\setminus V$ is disconnected for every nonempty relatively open $V\subseteq \Phi(a^*)$ and by Lemma~\ref{L4} we have a closed $b\subseteq \bbH$ such that $b^*=\Phi(a^*)$. 
\end{proof}

%Before proceeding to prove Theorem~\ref{T1} we show why it implies Theorem~\ref{T0}. 

The following result  together 
with Lemma~\ref{L-1} implies Theorem~\ref{T0}. 

\begin{thm}[PFA] \label{T1} 
If $X$ and $Y$ are locally compact, separable, metrizable spaces
then every homeomorphism $\Phi\colon X^*\to  Y^*$ that has a representation is trivial. 
\end{thm} 

\subsection{Proof
of   Theorem~\ref{T1}}
Fix $X, Y$ and $\Phi$ as in Theorem~\ref{T1}. 
Fix compact sets $K_n$ for $n\in \omega$ such that $X=\bigcup_n K_n$ and 
$K_n$ is included in the interior of $K_{n+1}$ for all $n$.  
For a topological space $Z$ let  
\[
\calD(Z)=\{a\subseteq Z: a\text{ is infinite, closed, and  discrete}\}. 
\]
We start with a straightforward application of Martin's Axiom 
(write $a\subseteq^*b$ if $a\setminus b$ is finite).   

\begin{lemma}[MA$_\kappa$] 
\label{L0} 
If  $X_0\subseteq X$ is countable.  
 and  $A\subseteq \calD(X_0)$ has cardinality $\kappa$ 
then there exists $a\in \calD(X_0)$ such that $b\subseteq^* a$ for all $b\in A$. 
\end{lemma} 

\begin{proof} 
Let $\cA$ be a subset of $\calD(X_0)$ of cardinality $\kappa$. Define poset $\bbP$ as follows. 
It has conditions of the form $(s,k,A)$ where 
 $k\in \omega$, $s\subseteq K_k$ is finite,  and  $A\subseteq \cA$ is finite. 
We let  $(s,k,A)$ extend  $(t,l,B)$ if
 $s=t\cup \left(\bigcup B\cap (K_k\setminus K_l)\right)$. 
Conditions with the same working part $(s,k)$ are clearly compatible. 
Since $X_0$ is countable  $\bbP$ 
is $\sigma$-centered. 
For every $a\in \cA$ the set of all $(s,k,A)$ such that $a\in A$ is 
dense and for every $n\in \omega$ the set of  $(s,k,A)$ such that $k\geq n$ is dense. 
 If $G$ is a filter intersecting these dense sets then  $a_G=\bigcup_{(s,k,A)\in G} s$
is as required. 
\end{proof}

\begin{lemma}[PFA] \label{L3} Assume $\Phi\colon X^*\to Y^*$ is as in Theorem~\ref{T1}. 
If $a\in \calD(X)$ then there is a map 
$h_a\colon a\to Y$ such that 
$\Phi(x)=(\beta h_a) (x)$
for every $x\in a^*$ and $b=h_a[a]$ is in $\calD(Y)$. 
Equivalently, for every $b\subseteq a$ we have 
\[
\Phi(b^*)=h_a[b]^*.
\] 
\end{lemma} 

\begin{proof} Fix $a\in \calD(X)$. 
Since $\Phi$ has a representation, there is $b_0\subseteq Y$ such that 
$\Phi(a^*)=b^*_0$. By PFA and the main result of 
 \cite{DoHa:Images} (see also \cite[\S4]{Fa:AQ}), 
 $b_0$ is homeomorphic to a direct sum of $\omega$ and a compact
set. By removing this compact set we obtain $b\subseteq b_0$ in $\calD(Y)$ 
such that $\Phi(a^*)=b^*$. 
By   \cite{ShSte:PFA} or 
\cite{Ve:OCA}, the restriction of $\Phi$ to every $a\in \calD(X)$ is trivial and we obtain 
the required map $h_a\colon a\to Y$. 
\end{proof} 

By Lemma~\ref{L3} applied with the roles of $X$ and $Y$ reversed 
for $b\in \calD(Y)$ we obtain $g_b\colon b\to X$ such that 
\[
\Phi^{-1}(y)=\beta g_b(y)
\]
for every $y\in b^*$. 

For $a\in \calD(X)$, a subset of $a$ is compact if and only if it is finite. 
We note the following immediate consequences. 
\begin{enumerate}
\item [(A)] If $a$ and $a'$ are in $\calD(X)$ then $h_a(x)\neq h_{a'}(x)$ for at most finitely 
many $x\in a\cap a'$. 
\item [(B)] If $a\in \calD(X)$ and $b=h_a[a]$, then $(g_b\circ h_a)(x)\neq x$ for at most 
finitely many $x\in a$. 
\end{enumerate}

%Write $X$ as an increasing union of compact sets
%$X=\bigcup_n K_n$ so that the interior of $K_{n+1}$ includes $K_n$ for all $n$. 

\begin{lemma}[PFA] \label{L.1-} Assume $X,Y$ and $\Phi$ are as in Theorem~\ref{T1}. 
If $X_0$ is a countable subset of $X$ with non-compact closure then 
there exists
 $h^{X_0} \colon X_0\to Y$ such that  
 $\Phi(a^*)=h^{X_0}[a]^*$ for every $a\in \calD(X_0)$. 
\end{lemma} 

\begin{proof} 
Let $Y\cup\{\infty\}$ denote the one-point compactification of $Y$. 
For $n\in \omega$ define a partition of unordered pairs in $\calD(X_0)$ 
   by 
\[
\{a,b\}\in K_0^n\text{ iff } (\exists x\in (a\cap b)\setminus K_n) h_a(x)\neq h_b(x). 
\]
Identify $a\in \calD(X_0)$ with a function $\tilde h_a\colon X_0\to Y\cup \{\infty\}$ 
that extends $h_a$ and sends $X_0\setminus a$ to $\infty$ 
and equip $(Y\cup \{\infty\})^{X_0}$ with the product topology. 
Since $Y\cup\{\infty\}$ is compact and metrizable, 
with this identification each $K_0^n$ is an open partition. 
For distinct $\alpha$ and $\beta$ in  $2^\omega$ 
we denote the least $n$ such that $\alpha(n)\neq \beta(n)$ by $\Delta(\alpha,\beta)$. 
By  PFA and the main result of  \cite[\S 3]{Fa:Cauchy}  
one of the two following possibilities \ref{I1} or \ref{I2} (corresponding 
to (b') and (a) of \cite[\S 3]{Fa:Cauchy}, respectively) applies. 

\subsubsection{There is  
$\cZ\subseteq 2^\omega$ of cardinality $\aleph_1$ and a continuous injection   
$\eta\colon \cZ\to \calD(X_0)$ such that $
\{\eta(\alpha), \eta(\beta)\}\in K_0^{\Delta(\alpha,\beta)}$
for all distinct $\alpha$ and $\beta$ in $\cZ$.}\label{I1} 

We shall prove that this alternative leads to contradiction. 
Since $\bigcup\eta[Z]\subseteq X_0$ and $X_0$ is countable,  
by Lemma~\ref{L0} we can find $a\in \calD(X_0)$ such that $\eta(\alpha)\subseteq^* a$ for 
all $\alpha\in \cZ$. 

Fix for a moment $\alpha\in \cZ$.  Then for all  but finitely many $y\in \eta(\alpha)$ 
we have  $h_{\eta(\alpha)}(y)=h_a(y)$. By a counting argument we can find $\bar m\in \omega$
and an uncountable $\cZ_0\subseteq \cZ$ such that for every $\alpha\in \cZ_0$ we have
$\eta(\alpha)\setminus K_{\bar m}\subseteq a$ and for all 
$y\in \eta(\alpha)\setminus K_{\bar m}$ we have $ h_{\eta(\alpha)}(y)=h_a(y)$. 

Pick $\alpha$ and $\beta$ in $\cZ_0$ such that $\Delta(\alpha,\beta)>\bar m$. 
Then $\{\eta(\alpha), \eta(\beta)\}\in K_0^{\bar m}$ and 
there exists $y\in (\eta(\alpha)\cap \eta(\beta))\setminus K_{\bar m}$
such that $h_{\eta(\alpha)}(y)\neq h_{\eta(\beta)}(y)$, contradicting the fact that 
both functions agree with $h_a$ past $K_{\bar m}$.

\subsubsection{There are $\cY_n$, for $n\in \omega$, such that 
$\calD(X_0)=\bigcup_n \cY_n$ and $[\cY_n]^2\cap K_0^n=\emptyset$ for all $n$.} \label{I2} 
Consider $\calD(X_0)$ as a partial ordering with respect to $\subseteq^*$.  
By Lemma~\ref{L0}, every countable subset of $\calD(X_0)$ is bounded. 
Therefore there exists $\bar n$ such that $\cY_{\bar n}$ is $\subseteq^*$-cofinal in $\calD(X_0)$
(see e.g., \cite[Lemma~2.2.2 (b)]{Fa:AQ}). 
For each $a\in \cY_{\bar n}$ let 
\[
\tilde h_a=h_a\rs (X\setminus K_{\bar n}). 
\]
Then 
$h=\bigcup\{\tilde h_a: a \in \cY_{\bar n}\}$
is a function since $\{a,b\}\notin K_0^{\bar n}$ for all distinct $a$ and $b$ in $\cY_{\bar n}$. 
Then for every $a\in \calD(X_0)$ there exists $b\in \cY_{\bar n}$ 
such that $a\subseteq^* b$. 
We therefore  have  $a\subseteq^* \dom(h)$ and 
$h_a(x)=h(x)$ for all but finitely many $x\in a$. 
This in particular implies that with $h^{X_0}$ as guaranteed by Lemma~\ref{L.1-} we have  
 $\Phi(a^*)=h^{X_0}[a]^*$ for every $a\in \calD(X_0)$. 
\end{proof} 

%Now we shall improve the map $h$ from Lemma~\ref{L1}. 

\begin{lemma}[PFA] \label{L1} Assume $X,Y$ and $\Phi$ are as in Theorem~\ref{T1}. 
If $X_0$ is a countable subset of $X$ with non-compact closure then 
there exist countable $Y_0\subseteq Y$, $m$,  and a homeomorphism  
 $h \colon X_0\setminus K_m\to Y_0$ such that  
 $\Phi(a^*)=h[a]^*$ for every $a\in \calD(X_0)$. 
\end{lemma} 

\begin{proof} First apply Lemma~\ref{L.1-} to $X_0$ and obtain $h^{X_0}$. 
Since the assumptions on $X,Y$ and $\Phi$ are symmetric, 
we can apply  Lemma~\ref{L.1-}
 with the roles of $X$ and $Y$ reversed and $Y_0=h^{X_0}[X_0]$ in place of $X_0$. 
We obtain a function $g$ from a co-compact subset of $Y_0$ into $X$ such that 
for every $a\in \calD(Y_0)$ the domain of $g$ includes $a$ modulo finite and 
$g_a(y)=g(y)$ for all but finitely many $y\in a$. Now define a new function $h$ to be 
the function whose graph is 
the intersection of the graphs of $h^{X_0}$ and $g^{-1}$. 
That is, $h^{X_0}(x)=y$ if $h(x)=y$ and $g(y)=x$, and undefined otherwise. 

We claim that $X_0\setminus \dom(h)$ is compact. Otherwise there exists 
$a\in \calD(X_0)$ disjoint from $\dom(h)$, but this contradicts the choice of $h^{X_0}$ and~$g$. 

We claim that for $a\in \calD(X_0)$ we have $h[a]\in \calD(Y)$. 
Since $a^*$ is nonempty, $h[a]$ is not compact. It therefore suffices to show that 
it has no infinite  subset whose closure is included in $Y$. 
 But if $b\subseteq h[a]$ were infinite and such that $b^*=\emptyset$, 
then $a_1=h^{-1}(b)$ would be a non-compact subset of $a$, contradicting
$\Phi(a_1^*)=b^*$. 

By removing compact sets from  $X_0$ and  $Y_0$ we may assume  
 that $\dom(h)=X_0$ and $\range(h)=Y_0$.

\begin{claim} There is $\bar k$ such that the restriction of  $h$ 
to $X_0\setminus K_{\bar k}$ is continuous
and the restriction of $h^{-1}$ to $Y_0\setminus L_{\bar k}$ is continuous. 
\end{claim} 

\begin{proof} 
Assume that the restriction of $h$ to $X_0\setminus K_{k}$ is discontinuous 
for all $k$. 
For every $n$ choose a 
sequence $\{x_{n,i}\}_i$ in $X_0\cap (K_{n+2}\setminus K_n)$
   converging to $x_n$ such that 
$\lim_i h(x_{n,i})\neq h(x_n)$.

Since 
 both $h$ and $h^{-1}$ send relatively compact sets to relatively compact sets and 
the interior of $K_{n+1}$ includes $K_n$ for all $n$, 
for every $m$ there exists $n$ such that $h[\dom(h)\cap K_m]\subseteq L_n$
and $h^{-1}[L_m]\subseteq K_n$. 
We can therefore go to a subsequence $n(j)$, for $j\in \omega$, 
 such that (with $n(0)=0$) for $j\geq 1$ and all $i$ we have 
  \[
  h(x_{n(j),i})\in L_{n(j+1)}\setminus L_{n(j-1)}
\qquad\text{ and } 
  h(x_{n(j)})\in L_{n(j+1)}\setminus L_{n(j-1)}. 
\]
The only accumulation points of the set 
\[
c=\{x_{n(j)}, x_{n(j),i}: j\geq 1, i\in\omega\}
\]
 are 
$x_{n(j)}$, for $j\geq 1$. Since $x_{n(j)}$, for $j\geq 1$, form a closed discrete set 
   $c$ is homeomorphic to  
$\omega^2$ equipped with its ordinal topology. 
The proof can now be completed by applying the weak Extension Principle (wEP) of \cite[\S 4]{Fa:AQ}, but we give
an elementary and self-contained proof. 

Fix a nonprincipal ultrafilter $\cU$ on $\omega$. Then $y_j=\lim_{i\to \cU} h(x_{n(j), i})$ exists and is  in $L_{n(j+1)}$ by compactness. 
Since each $L_{n(j+1)}$ is second countable, by a diagonal argument we can choose a sequence $i(k)$, for $k\in \omega$, such that 
\[
\lim_k h(x_{n(j),i(k)})=y_j
\]
for all $j$. 
Since $y_j$ and $h(x_{n(j)})$  are distinct elements of $L_{n(j+1})\setminus L_{n(j-1)}$ for all $j$, 
we can find disjoint  open subsets $U$ and $V$ of $Y$ such that 
$h(x_{n(j)})\in U$ and $y_j\in V$ for all $j$. By going to subsequences again and re-enumerating  we can assume that 
$h(x_{n(j),i(k)})\in U$ for all $j,k$.  
If $W$ and $S$ are disjoint open subsets of $X$ such that $\Phi(W^*)=U^*$ and $\Phi(S^*)=V$ 
then we have that 
$x_{n(j)}\in W$ for all but finitely many $j$ but $x_{n(j),i(k)}\notin S$ for  
 every $j$ and all but finitely many $k$---a contradiction. 
\end{proof}

By the above argument, for every countable $X_0\subseteq X$ we can find $n$ and a continuous  
function $h^{X_0}\colon X_0\setminus K_n\to Y$ such that $\beta h^{X_0}$ 
agrees with $\Phi$ on $X_0^*$. 
If $X_0\subseteq X_1$ then $h^{X_1}$ extends $h^{X_0}\rs (X_0\setminus K_n)$ for 
a large enough $n$. 
\end{proof}

  \begin{lemma} \label{L0.1} Assume $X, Y$ and $\Phi$  are as in the assumption of 
  Theorem~\ref{T1} and 
$\Phi_1\colon X^*\to Y^*$ is a trivial 
homeomorphism such that $\Phi^{-1}$ and $\Phi_1^{-1}$ agree on sets of the form 
$a^*$ for $a\in \calD(Y)$. 
  Then $\Phi=\Phi_1$. 
  \end{lemma} 

\begin{proof} 
Fix  a representation $\Psi$ of $\Phi$ and a 
 homeomorphism $h\colon X\setminus K\to Y\setminus L$ 
   between co-compact subsets of $X$ and $Y$ such that $(\beta h)[F^*]=\Phi_1(F^*)$
  for all $F\in \cF_X$. 
Assume $\Phi\neq \Phi_1$. Then for some $F\in \cF_X$ we have that 
$\Psi(F)\Delta h[F]$ is not compact. We can therefore find $a\in \calD(Y)$ such that 
(i) $a\subseteq \Psi(F)\setminus h[F]$ or (ii) $a\subseteq h[F]\setminus \Psi(F)$. 

In either case we have that $\Phi^{-1}(a^*)\cap \Phi_1^{-1}(a^*)=\emptyset$,
contradicting our assumption. 
\end{proof}

  \begin{proof}[Proof of Theorem~\ref{T1}]
Fix a countable dense  set $X_0\subseteq X$ and apply Lemma~\ref{L1}
to obtain $h^{X_0}$. Let $\tilde h$ be the maximal continuous extension of $h^{X_0}$ to 
a $G_\delta$ subset $X_1$ of $X$. We claim that $X_1\supseteq X_0\setminus K_n$ for 
some $n$. Otherwise, find $a\in \calD(X)$ disjoint from $X_1$ and apply Lemma~\ref{L1} 
to $X_2=X_0\cup a$. The resulting continuous function $h^{X_2}$ agrees with $h^{X_0}$ 
on $X_0\setminus K_n$ for a large enough $n$. Since $K_n$ is included in the interior of $K_{n+1}$, 
 the restriction of $h^{X_2}$ to $\dom(h^{X_2}\setminus K_{n+1})$ is compatible with $\tilde h$
 contradicting the assumption that $h^{X_0}$ 
cannot be continuously extended to the points in~$a$. 

Therefore the domain of $\tilde h$ contains $X\setminus K_n$ for a large enough $n$. 
The analogous argument shows that the range of $\tilde h$ includes $Y\setminus K_m$ for 
a large enough $m$, and that $\tilde h$ is a homeomorphism. 
The restriction of the map $\beta\tilde h$ to $X^*$ is a homeomorphism between $X^*$ and $Y^*$, 
and by Lemma~\ref{L0.1} this trivial homeomorphism coincides with $\Phi$. 
\end{proof}

An autohomeomorphism $\Phi$ of  $(X^*)^\kappa$ is \emph{trivial} 
if there are a permutation $\sigma$ of $\kappa$ and autohomeomorphisms 
$f_\xi$, $\xi<\kappa$, of $\bbH$ such that $\Phi(x)(\xi)=\beta f_\xi(x(\sigma(\xi))$ for every $x\in (X^*)^\kappa$. 
Since $\bbH^*$ is connected, the following is an immediate consequence of Theorem~\ref{T0} and the main result of~\cite{Fa:Dimension}. 

\begin{cor}[PFA] For an arbitrary cardinal $\kappa$, 
 all autohomeomorphisms  of $(\bbH^*)^\kappa$ are trivial. \qed
 \end{cor}

\subsection{A homeomorphism without a representation}
\label{S.Norep}
It is now time to give an example promised in the beginning of \S\ref{S.Rep}. 
By Parovi\v{c}enko's theorem,  CH implies that 
 $(\omega^2)^*$ (where $\omega^2$ is taken with respect to the 
ordinal topology) and $\omega^*$ are 
homeomorphic. However, a homeomorphism $\Phi\colon  (\omega^2)^*\to \omega^*$ does not have a
representation. If $a\subseteq\omega^2$ is the set of limit ordinals below $\omega^2$, then $a^*$ 
is closed and nowhere dense. Therefore $\Phi(a^*)$ is a closed nowhere dense subset of $\omega^*$. 
Since for every $b\subseteq\omega$ the set $b^*$ is clopen, $\Phi(a^*)\neq b^*$ for all $b$.

On the other hand, $\Phi^{-1}$ has a representation. 
As a matter of fact, whenever $\Psi\colon \omega^*\to X^*$ is a homeomorphism 
then $\Psi$ has a representation. In order to show this it suffices to prove that
if $F\subseteq X^*$ is clopen then $F=b^*$ for closed $b\subseteq X$. 
But if $U$ and $V$ are open subsets of $\beta X$ such that $U\cap X^*=F$ and $X^*\setminus V=F$, 
then clearly $b=X\setminus V$ satisfies $b^*=F$.

\section{Concluding remarks} 
\label{S.CR} 
The motivation for this work comes 
from \cite[Conjecture~1.2 and Conjecture~1.3]{CoFa:Automorphisms}. We restate the abelian case of these conjectures in its dual form. 

\begin{conj} [PFA] Every homeomorphism between Stone--\v{C}ech remainders
of locally compact Polish spaces $X$ and $Y$ is trivial. 
\end{conj}

Since every trivial homeomorphism has a representation, by Theorem~\ref{T1}
this is equivalent to conjecture that under PFA 
every homeomorphism $\Phi\colon X^*\to Y^*$ between 
remainders of  locally compact, non-compact, 
Polish spaces $X$ and $Y$ 
has a representation.

\begin{conj}\label{C.CH}   Continuum Hypothesis implies that 
$X^*$ has $2^{\aleph_1}$ nontrivial autohomeomorphisms for every locally compact, non-compact, 
separable metrizable space $X$.
\end{conj}  

By Proposition~\ref{P.NS} (5), $C(X^*)$ is not countably quantifier-free saturated
for some locally compact Polish spaces $X$. However, the space constructed there 
includes a copy of $\bbH$ as a clopen subset and therefore $C(X^*)$ has 
at least as many automorphisms as $C(\bbH^*)$. 
 Large families of automorphisms of coronas 
that are not countably saturated were constructed in \cite{CoFa:Automorphisms} using the Continuum Hypothesis. 
We do not know whether Conjecture~\ref{C.CH} is true for $X=\bbR^{n+1}$, for $n\geq 1$. 
It may be  worth mentioning that for $n\geq 1$ we have (with $\bbT$ denoting the unit circle)
 \[
C((\bbR^{n+1})^*)
\cong \Cb(\bbH, C(\bbT^n))/C_0(\bbH,C(\bbT^n)). 
\]
To see this, remove a small open ball containing the origin from $\bbR^{n+1}$
and note that $\Cb(\bbH\times \bbT^n)\cong \Cb(\bbH, C(\bbT^n))$
(this follows from \cite[3.4]{akemann1973multipliers}, see also 
\cite[II.7.3.12 (iv)]{Black:Operator} by noting that in a unital algebra norm and strict topologies coincide). 

Thus a relevant question is  to what extent
 the assumption on compactness of domains in $A$ 
can be removed from Theorem~\ref{T2}? Proposition~\ref{P.NS} (5) gives a warning sign. 

By Woodin's $\Sigma^2_1$ absoluteness theorem (see \cite{Wo:Beyond}), 
   Continuum Hypothesis is the optimal set-theoretic assumption 
for obtaining  autohomeomorphisms as in Conjecture~\ref{C.CH}. 

\def\germ{\frak} \def\scr{\cal} \ifx\documentclass\undefinedcs
  \def\bf{\fam\bffam\tenbf}\def\rm{\fam0\tenrm}\fi % f**k-amstex!
  \def\defaultdefine#1#2{\expandafter\ifx\csname#1\endcsname\relax
  \expandafter\def\csname#1\endcsname{#2}\fi} \defaultdefine{Bbb}{\bf}
  \defaultdefine{frak}{\bf} \defaultdefine{=}{\B} % doublef**k-amstex!!
  \defaultdefine{mathfrak}{\frak} \defaultdefine{mathbb}{\bf}
  \defaultdefine{mathcal}{\cal}
  \defaultdefine{beth}{BETH}\defaultdefine{cal}{\bf} \def\bbfI{{\Bbb I}}
  \def\mbox{\hbox} \def\text{\hbox} \def\om{\omega} \def\Cal#1{{\bf #1}}
  \def\pcf{pcf} \defaultdefine{cf}{cf} \defaultdefine{reals}{{\Bbb R}}
  \defaultdefine{real}{{\Bbb R}} \def\restriction{{|}} \def\club{CLUB}
  \def\w{\omega} \def\exist{\exists} \def\se{{\germ se}} \def\bb{{\bf b}}
  \def\equivalence{\equiv} \let\lt< \let\gt>
\providecommand{\bysame}{\leavevmode\hbox to3em{\hrulefill}\thinspace}
\providecommand{\MR}{\relax\ifhmode\unskip\space\fi MR }
% \MRhref is called by the amsart/book/proc definition of \MR.
\providecommand{\MRhref}[2]{%
  \href{http://www.ams.org/mathscinet-getitem?mr=#1}{#2}
}
\providecommand{\href}[2]{#2}

% \bibliography{lista,listb,ifmain}
%\bibliographystyle{amsplain}

\end{document}